\documentclass[11pt]{amsart}

\usepackage[utf8]{inputenc}
\usepackage{amsmath,amssymb,amsthm,mathtools}
\usepackage[margin=1.2in]{geometry}
\usepackage{enumitem}
\usepackage[hidelinks]{hyperref}

\mathtoolsset{showonlyrefs}

\newtheorem{theorem}{Theorem}[section]
\newtheorem{proposition}[theorem]{Proposition}
\newtheorem{lemma}[theorem]{Lemma}
\newtheorem{corollary}[theorem]{Corollary}
\theoremstyle{definition}
\newtheorem{remark}[theorem]{Remark}

\numberwithin{equation}{section}

\begin{document}

\title[Affine variational functionals, isoperimetric problems and stability]{Affine variational functionals, isoperimetric problems and stability}

\author{Francisco Mar\'in Sola}
\address{Department of Engineering and Technology of Computers, area of Applied Mathematics. Universidad de Murcia, Spain.}
\email{francisco.marin7@um.es}

\thanks{The author is supported by the grant ``Proyecto financiado por la CARM a trav\'es de la convocatoria de Ayudas a proyectos para el desarrollo de investigaci\'on cient\'ifica y t\'ecnica por grupos competitivos, incluida en el Programa Regional de Fomento de la Investigaci\'on Cient\'ifica y T\'ecnica (Plan de Actuaci\'on 2022) de la Fundaci\'on S\'eneca-Agencia de Ciencia y Tecnolog\'ia de la Regi\'on de Murcia, REF. 21899/PI/22''. This work is also associated with grant PID2025-169523NB-I00, provisionally proposed for funding by MICIU/AEI under the 2025 “Knowledge Generation Projects” call}


\begin{abstract}

We study three affine variational quantities built on the affine $L^p$ energy of Lutwak, Yang and Zhang: the affine Poincar\'e--Sobolev constants $\Lambda^{\mathcal{A}}_{p,q}(\Omega)$ of a bounded open set, the associated affine torsional rigidity, and the affine $p$-capacity of a compact set. We show that ellipsoids are the only extremal sets in their corresponding isoperimetric inequalities, and that the affine capacity of a compact set of positive volume is attained by a potential solving the affine $p$-Laplace equation. We also prove an affine Kohler-Jobin inequality: among sets with prescribed affine torsional rigidity, only ellipsoids minimize $\Lambda^{\mathcal{A}}_{p,q}$. Following the approach of Haddad, Jim\'enez and Montenegro to affine Sobolev inequalities, our methods rest on the $L_p$ Busemann--Petty centroid inequality. Proving a stability version of it, we obtain quantitative forms of the isoperimetric inequalities for $\Lambda^{\mathcal{A}}_{p,q}$, the affine torsional rigidity and the affine $p$-capacity, together with a stability result for the affine Kohler-Jobin inequality on convex sets.
\end{abstract}

\maketitle

\section{Introduction}

For $1<p<n$, the sharp Sobolev inequality states that
\begin{equation}\label{eq:Sobolev}
S_{n,p}\,\|f\|_{L^{p^*}(\mathbb{R}^n)}\leq\|\nabla f\|_{L^p(\mathbb{R}^n)},\qquad p^*=\frac{np}{n-p},
\end{equation}
for every $f\in W^{1,p}(\mathbb{R}^n)$; the optimal constant $S_{n,p}$ and the extremal functions were determined by Aubin \cite{Aubin1976} and Talenti \cite{Talenti1976}. On a bounded open set $\Omega\subset\mathbb{R}^n$ the Dirichlet energy also controls the $L^p$ norm: the Poincar\'e inequality states that $\|\nabla v\|_{L^p(\Omega)}\geq c\,\|v\|_{L^p(\Omega)}$ for every $v\in W^{1,p}_0(\Omega)$, with $c>0$ depending on $n$, $p$ and $\Omega$. Together, the two inequalities control every subcritical norm, and the optimal constants
\begin{equation}\label{eq:PS}
\Lambda_{p,q}(\Omega)=\inf\Big\{\int_\Omega|\nabla v|^p\,dx\ :\ v\in W^{1,p}_0(\Omega),\ \|v\|_{L^q(\Omega)}=1\Big\},\qquad 1\leq q<p^*,
\end{equation}
where $p^*=\infty$ if $p\geq n$, are the Poincar\'e--Sobolev constants of $\Omega$. Throughout the paper $p>1$, and $1<p<n$ whenever capacities are involved.

The constants \eqref{eq:PS} are the basic variational functionals associated with the $p$-Laplace operator $\Delta_pu=\operatorname{div}(|\nabla u|^{p-2}\nabla u)$. The infimum in \eqref{eq:PS} is attained, and every nonnegative minimizer $u$ solves $-\Delta_pu=\Lambda_{p,q}(\Omega)\,u^{q-1}$ in $\Omega$ with $u=0$ on $\partial\Omega$. For $q=p$, $\Lambda_{p,p}(\Omega)$ is the first eigenvalue of the Dirichlet $p$-Laplacian. For $q=1$, the quantity $T_p(\Omega)=\Lambda_{p,1}(\Omega)^{-1/(p-1)}$ is the $p$-torsional rigidity of $\Omega$, and $T_p(\Omega)=\int_\Omega u_\Omega\,dx$, where $u_\Omega$ solves $-\Delta_pu=1$ in $\Omega$ with $u=0$ on $\partial\Omega$ (see \cite[Proposition 2.2]{Brasco2014}). By the P\'olya--Szeg\H{o} principle, balls are optimal at fixed volume: if $\Omega^\star$ denotes the open ball centered at the origin with $|\Omega^\star|=|\Omega|$, then
\begin{equation}\label{eq:FK}
\Lambda_{p,q}(\Omega)\geq\Lambda_{p,q}(\Omega^\star).
\end{equation}
For $q=p$ this is the Faber--Krahn inequality, and for $q=1$ it is the Saint-Venant inequality $T_p(\Omega)\leq T_p(\Omega^\star)$, proved by P\'olya \cite{Polya1948} for $p=2$ (see also \cite{PolyaSzego1951}). P\'olya and Szeg\H{o} \cite{PolyaSzego1951} conjectured a stronger statement for $p=2$: among sets with given torsional rigidity, balls minimize the first eigenvalue. Kohler-Jobin proved it by introducing a new rearrangement \cite{KohlerJobin1975,KohlerJobin1978,KohlerJobin1982}, and Brasco \cite{Brasco2014} extended her method to every $p>1$ and every $1<q<p^*$, proving that
\begin{equation}\label{eq:KJEuclid}
T_p(\Omega)^{(p-1)\sigma}\,\Lambda_{p,q}(\Omega)\geq T_p(B^n)^{(p-1)\sigma}\,\Lambda_{p,q}(B^n),\qquad\sigma=\frac{\frac{p}{n}+\frac{p}{q}-1}{\frac{p}{n}+p-1},
\end{equation}
with equality only for balls. Here $B^n$ is the Euclidean unit ball, and the exponent makes the left-hand side invariant under dilations of $\Omega$. Since balls maximize $T_p$ at fixed volume, \eqref{eq:KJEuclid} together with the Saint-Venant inequality implies \eqref{eq:FK} for $q>1$ \cite[Remark 1.2]{Brasco2014}. The stability of \eqref{eq:FK} is measured by the Fraenkel asymmetry
\[
\mathcal{F}(S)=\inf\Big\{\frac{|S\triangle B|}{|S|}\ :\ B\text{ a ball with }|B|=|S|\Big\}
\]
of a set $S$ of positive finite measure. Fusco, Maggi and Pratelli \cite{FMP2009} proved that the deficit in \eqref{eq:FK} controls a power of $\mathcal{F}(\Omega)$ for every $p>1$ and $1\leq q<p^*$, and the sharp form was established by Brasco, De Philippis and Velichkov \cite{BDV2015} for $p=2$ and by Fusco and Zhang \cite{FuscoZhang2017} for every $p>1$: there is $c=c(n,p,q)>0$ such that
\begin{equation}\label{eq:quantFK}
\Lambda_{p,q}(\Omega)\geq\big(1+c\,\mathcal{F}(\Omega)^2\big)\Lambda_{p,q}(\Omega^\star)
\end{equation}
for every open set $\Omega$ of finite measure. We refer to the survey by Brasco and De Philippis \cite{BrascoDePhilippis2017} for the history of these results. To the best of our knowledge there is no quantitative form of \eqref{eq:KJEuclid}.

For $1<p<n$, the same energy governs the $p$-capacity of a compact set $K\subset\mathbb{R}^n$,
\begin{equation}\label{eq:cap}
\operatorname{cap}_p(K)=\inf\Big\{\int_{\mathbb{R}^n}|\nabla u|^p\,dx\ :\ u\in C^\infty_c(\mathbb{R}^n),\ u\geq1\ \text{on }K\Big\}.
\end{equation}
The infimum is attained in the homogeneous Sobolev space $D^{1,p}(\mathbb{R}^n)$ by a unique function, the capacitary potential of $K$. It is $p$-harmonic in $\mathbb{R}^n\setminus K$, and $-\operatorname{div}(|\nabla u|^{p-2}\nabla u)$ is a nonnegative measure concentrated on $\partial K$, the equilibrium measure of $K$, whose total mass is $\operatorname{cap}_p(K)$. We refer to the books by Maz'ya \cite{Mazya2011} and by Heinonen, Kilpel\"ainen and Martio \cite{HKM2006} for background.

Among compact sets of given volume, balls have the least capacity: if $K^\star$ denotes the closed ball centered at the origin with $|K^\star|=|K|$, then
\begin{equation}\label{eq:isocap}
\operatorname{cap}_p(K)\geq\operatorname{cap}_p(K^\star).
\end{equation}
This isocapacitary inequality also follows from the P\'olya--Szeg\H{o} principle. Its stability was conjectured by Hall, Hayman and Weitsman \cite{HHW1991} in terms of the Fraenkel asymmetry $\mathcal{F}(K)$. Fusco, Maggi and Pratelli \cite{FMP2009} also proved a first quantitative version, and the sharp one was established by De Philippis, Marini and Mukoseeva \cite{DPMM2021} for $p=2$ and by Mukoseeva \cite{Mukoseeva2023} for $1<p<n$: there is $c=c(n,p)>0$ such that
\begin{equation}\label{eq:quantisocap}
\operatorname{cap}_p(K)\geq\big(1+c\,\mathcal{F}(K)^2\big)\operatorname{cap}_p(K^\star),
\end{equation}
and the exponent $2$ cannot be improved. In \cite{DPMM2021,Mukoseeva2023} the inequality is stated for open sets; see Section \ref{sec:functions} for compact sets.

The affine counterpart of \eqref{eq:Sobolev} was discovered by Zhang \cite{Zhang1999} for $p=1$ and by Lutwak, Yang and Zhang \cite{LYZ2002} for $1<p<n$. For $f\in W^{1,p}(\mathbb{R}^n)$ consider the affine energy
\begin{equation}\label{eq:Ep}
\mathcal{E}_p(f)=c_{n,p}\Big(\int_{S^{n-1}}\|\nabla_\xi f\|_{L^p(\mathbb{R}^n)}^{-n}\,d\xi\Big)^{-1/n},
\end{equation}
where $\nabla_\xi f=\langle\nabla f,\xi\rangle$ and the constant $c_{n,p}$ is chosen so that $\mathcal{E}_p(f)=\|\nabla f\|_{L^p(\mathbb{R}^n)}$ whenever $f$ is radial (see Section \ref{sec:energy}). The affine Sobolev inequality of \cite{LYZ2002} reads
\begin{equation}\label{eq:affineSobolev}
S_{n,p}\,\|f\|_{L^{p^*}(\mathbb{R}^n)}\leq\mathcal{E}_p(f)\leq\|\nabla f\|_{L^p(\mathbb{R}^n)},
\end{equation}
where the right inequality follows from H\"older's inequality and is strict unless the directional energies $\|\nabla_\xi f\|_{L^p}$, $\xi\in S^{n-1}$, all coincide. The affine energy is invariant under volume-preserving affine transformations of $\mathbb{R}^n$, whereas $\|\nabla f\|_{L^p}$ is only invariant under rigid motions; thus the left inequality in \eqref{eq:affineSobolev} is stronger than \eqref{eq:Sobolev}, with the same optimal constant. One of the main tools of the theory is the affine P\'olya--Szeg\H{o} principle
\begin{equation}\label{e:affinePZ}
    \mathcal{E}_p(f^\star)\leq\mathcal{E}_p(f),
\end{equation}
where $f^\star$ is the symmetric decreasing rearrangement of $|f|$, due to Cianchi, Lutwak, Yang and Zhang \cite{CLYZ2009} (see also \cite{HSX2012}); its equality cases were determined by Wang \cite{Wang2013} and, for a larger family of affine energies, by Nguyen \cite{Nguyen2016}, and a quantitative form was obtained in \cite{Wang2013}. A different approach to sharp affine Sobolev-type inequalities, through the $L_p$ Busemann--Petty centroid inequality, was developed by Haddad, Jim\'enez and Montenegro \cite{HJM2016}.

On bounded domains, the same authors proved in \cite{HJM2021} an affine Poincar\'e inequality showing that $\mathcal{E}_p$ controls the $L^p$ norm of functions in $W^{1,p}_0(\Omega)$. This leads to the affine Poincar\'e--Sobolev constants
\begin{equation}\label{eq:affinePS}
\Lambda^{\mathcal{A}}_{p,q}(\Omega)=\inf\Big\{\mathcal{E}_p(v)^p\ :\ v\in W^{1,p}_0(\Omega),\ \|v\|_{L^q(\Omega)}=1\Big\},\qquad 1\leq q<p^*,
\end{equation}
and to the affine torsional rigidity $T^{\mathcal{A}}_p(\Omega)=\Lambda^{\mathcal{A}}_{p,1}(\Omega)^{-1/(p-1)}$. By \eqref{eq:affineSobolev}, $\Lambda^{\mathcal{A}}_{p,q}(\Omega)\leq\Lambda_{p,q}(\Omega)$ and $T^{\mathcal{A}}_p(\Omega)\geq T_p(\Omega)$. Hence every lower bound for $\Lambda^{\mathcal{A}}_{p,q}$, and every upper bound for $T^{\mathcal{A}}_p$, in terms of the volume improves the corresponding Euclidean inequality. Both quantities are invariant under volume-preserving affine maps, so that ellipsoids play the role of balls.

The first variation of \eqref{eq:affinePS} at $q=p$ led Haddad, Jim\'enez and Montenegro \cite{HJM2021} to the affine $p$-Laplace operator. For $u\in W^{1,p}_0(\Omega)\setminus\{0\}$, and for $u\in D^{1,p}(\mathbb{R}^n)\setminus\{0\}$ when $1<p<n$, it is defined by
\begin{equation}\label{eq:affineLaplacian}
\Delta^{\mathcal{A}}_pu=-\operatorname{div}\big(H_u(\nabla u)^{p-1}\nabla H_u(\nabla u)\big),
\end{equation}
where
\[
H_u(\zeta)^p=c_{n,p}^{-n}\,\mathcal{E}_p(u)^{n+p}\int_{S^{n-1}}\|\nabla_\xi u\|_{L^p}^{-n-p}\,|\langle\xi,\zeta\rangle|^p\,d\xi,\qquad\zeta\in\mathbb{R}^n.
\]
For each $u$, $H_u$ is a norm on $\mathbb{R}^n$, namely the support function of a convex body built from $u$ (see Lemma \ref{lem:frozen}); hence $\Delta^{\mathcal{A}}_p$ is a nonlocal quasilinear operator which, at each $u$, acts as the anisotropic $p$-Laplacian of a norm determined by $u$ itself. It satisfies $\Delta^{\mathcal{A}}_p(u\circ A)=(\Delta^{\mathcal{A}}_pu)\circ A$ for $A\in\mathrm{SL}(n)$, and $\Delta^{\mathcal{A}}_pu=-\Delta_pu$ for radial $u$; for $p=2$ it is the affine Laplacian of Schindler and Tintarev \cite{SchindlerTintarev2018}. In \cite{HJM2021}, among other related results, it is proved that $\Lambda^{\mathcal{A}}_{p,p}(\Omega)$ is attained, that its minimizers are precisely the first eigenfunctions of $\Delta^{\mathcal{A}}_p$, which are bounded, of class $C^{1,\alpha}$ and of constant sign when $\Omega$ is connected, and that the affine Faber--Krahn inequality $\Lambda^{\mathcal{A}}_{p,p}(\Omega)\geq\Lambda_{p,p}(\Omega^\star)$ holds, with equality only for ellipsoids. Leite and Montenegro have extensively studied both the affine energy and $\Delta^{\mathcal{A}}_p$ \cite{LM2022,LM2024ACV}. Of particular relevance to this work, they proved that $\mathcal{E}_p^p$ is of class $C^1$ away from the origin and obtained existence results for equations driven by $\Delta^{\mathcal{A}}_p$ \cite{LM2024CV}. Langharst and Roysdon \cite{LR2025} extended the eigenvalue theory, including the affine Faber--Krahn inequality, to the so-called $m$-th order and asymmetric affine energies.

Our first main result concerns the isoperimetric problems for the whole family $\Lambda^{\mathcal{A}}_{p,q}$.

\begin{theorem}\label{thm:iso}
Let $p>1$, $1\leq q<p^*$, and let $\Omega\subset\mathbb{R}^n$ be a bounded open set. Then
\begin{equation}\label{eq:iso}
\Lambda^{\mathcal{A}}_{p,q}(\Omega)\geq\Lambda_{p,q}(\Omega^\star),
\end{equation}
with equality if and only if $\Omega$ is an ellipsoid, up to a set of zero \emph{Sobolev} $p$-capacity.
\end{theorem}
Here we say that $\Omega$ is an ellipsoid up to a set of zero Sobolev $p$-capacity $C_p$ (see Section \ref{sec:functions}) if $\Omega\subseteq E$ and $C_p(E\setminus\Omega)=0$ for some open ellipsoid $E$; such modifications do not change $W^{1,p}_0(\Omega)$ (see Section \ref{sec:functions}). For $q=p$, Theorem \ref{thm:iso} is the affine Faber--Krahn inequality \cite[Theorem 5]{HJM2021}. The new part is the characterization of the equality cases for $q\neq p$. The critical exponent is excluded: for $1<p<n$, $\Lambda^{\mathcal{A}}_{p,p^*}(\Omega)=S_{n,p}^p$ for every $\Omega$, and this value is not attained \cite{LM2024ACV}. For $q=1$, Theorem \ref{thm:iso} yields the following affine Saint-Venant inequality.

\begin{corollary}\label{cor:SV}
Let $p>1$ and let $\Omega\subset\mathbb{R}^n$ be a bounded open set. Then
\[
T_p(\Omega)\leq T^{\mathcal{A}}_p(\Omega)\leq T_p(\Omega^\star)=\frac{(p-1)\,\omega_n^{-\frac{p}{n(p-1)}}}{n^{\frac{1}{p-1}}\,\big(n(p-1)+p\big)}\,|\Omega|^{\frac{n(p-1)+p}{n(p-1)}},
\]
where $\omega_n=|B^n|$. Equality holds in the second inequality if and only if $\Omega$ is an ellipsoid, up to a set of zero Sobolev $p$-capacity.
\end{corollary}

With the definitions of the affine Poincar\'e--Sobolev constant and the affine torsional rigidity in hand, it is natural to wonder about an \emph{affine} analogue of \eqref{eq:KJEuclid}. Our next result provides such an inequality.

\begin{theorem}\label{thm:KJ}
Let $p>1$, $1<q<p^*$, and let $\Omega\subset\mathbb{R}^n$ be a bounded open set. Then
\begin{equation}\label{eq:KJ}
T^{\mathcal{A}}_p(\Omega)^{(p-1)\sigma}\,\Lambda^{\mathcal{A}}_{p,q}(\Omega)\geq T_p(B^n)^{(p-1)\sigma}\,\Lambda_{p,q}(B^n),
\end{equation}
with $\sigma$ as in \eqref{eq:KJEuclid}. Equality holds if and only if $\Omega$ is an ellipsoid, up to a set of zero Sobolev $p$-capacity.
\end{theorem}
Since $T^{\mathcal{A}}_p\geq T_p$ and $\Lambda^{\mathcal{A}}_{p,q}\leq\Lambda_{p,q}$, inequality \eqref{eq:KJ} is not, in general, a strengthening of \eqref{eq:KJEuclid}. Together with Corollary \ref{cor:SV}, Theorem \ref{thm:KJ} yields \eqref{eq:iso} for $q>1$ with an improved constant (see Remark \ref{rem:hierarchy}).

On the capacitary side, Xiao \cite{Xiao2016} introduced the capacity associated with $\mathcal{E}_p$. In our normalization, the affine $p$-capacity of a compact set $K$ is
\begin{equation}\label{eq:affcap}
\operatorname{cap}^{\mathcal{A}}_p(K)=\inf\Big\{\mathcal{E}_p(u)^p\ :\ u\in C^\infty_c(\mathbb{R}^n),\ u\geq1\ \text{on }K\Big\}.
\end{equation}
The affine capacity is invariant under volume-preserving affine maps and satisfies $\operatorname{cap}^{\mathcal{A}}_p\leq\operatorname{cap}_p$. Xiao proved the affine isocapacitary inequality
\begin{equation}\label{eq:affisocap}
\operatorname{cap}^{\mathcal{A}}_p(K)\geq\operatorname{cap}_p(K^\star),
\end{equation}
which strengthens \eqref{eq:isocap}, and showed that ellipsoids attain equality \cite[Theorems 3.2 and 3.5]{Xiao2016}; in particular $\operatorname{cap}^{\mathcal{A}}_p$ and $\operatorname{cap}_p$ coincide on balls. 

Further properties were obtained in \cite{Xiao2017}, and relative and BV versions of the affine capacity were studied in \cite{XiaoZhang2016} and \cite{WangXiao2020}. For other generalizations, the interested reader is referred to \cite{HongYe2018,ZhouYe2025}. Wang and Xiao \cite[Theorem 2.3]{WangXiao2020} computed the first variation of the affine capacity, assuming that a minimizer exists among compactly supported functions. Nevertheless, several basic questions have remained open in the literature: whether the infimum in \eqref{eq:affcap} is attained, which condition a minimizer satisfies on $K$, which sets attain equality in \eqref{eq:affisocap}, and whether \eqref{eq:affisocap} is stable. Our second main contribution answers the first three questions. 

\begin{theorem}\label{thm:potential}
Let $1<p<n$ and let $K\subset\mathbb{R}^n$ be a compact set with $|K|>0$. There exists $u\in D^{1,p}(\mathbb{R}^n)$ with $\tilde u\geq1$ quasi-everywhere on $K$ and $\mathcal{E}_p(u)^p=\operatorname{cap}^{\mathcal{A}}_p(K)$. Every such minimizer satisfies $0\leq u\leq1$ and $\tilde u=1$ quasi-everywhere on $K$, and it is the unique capacitary potential of $K$ for the anisotropic energy $v\mapsto\int_{\mathbb{R}^n}H_u(\nabla v)^p\,dx$. Moreover, $\Delta^{\mathcal{A}}_pu=0$ in $\mathbb{R}^n\setminus K$, and $\mu_K=\Delta^{\mathcal{A}}_pu$ is a nonnegative Radon measure supported on $\partial K$ with $\mu_K(\mathbb{R}^n)=\operatorname{cap}^{\mathcal{A}}_p(K)$.
\end{theorem}

Here $\tilde u$ denotes the quasicontinuous representative of $u$, quasi-everywhere means outside a set of zero Sobolev p-capacity, and the equations are understood in the sense of distributions. As one of the consequences of Theorem \ref{thm:potential}, we can derive the corresponding rigidity statement for \eqref{eq:affisocap}.

\begin{theorem}\label{thm:isocap}
Let $1<p<n$ and let $K\subset\mathbb{R}^n$ be a compact set with $|K|>0$. Then
\[
\operatorname{cap}_p(K)\geq\operatorname{cap}^{\mathcal{A}}_p(K)\geq\operatorname{cap}_p(K^\star),
\]
and equality holds in the second inequality  if and only if $K$ is an ellipsoid up to a set of zero Sobolev $p$-capacity, that is, $E\subseteq K$ and $C_p(K\setminus E)=0$ for some closed ellipsoid $E$.
\end{theorem}

The common thread of this paper is the variational structure of the affine energy. For every nonzero function $f$ in $W^{1,p}_0(\Omega)$, or in $D^{1,p}(\mathbb{R}^n)$ when $1<p<n$,
\[
\mathcal{E}_p(f)^p=\min\Big\{\int_{\mathbb{R}^n}h(\Gamma_pL,\nabla f)^p\,dx\ :\ L\text{ a star body with }|L|=\omega_n\Big\},
\]
where $\Gamma_pL$ is the $L_p$ centroid body of $L$ and $h(\Gamma_pL,\cdot)$ its support function (see Section \ref{sec:convex}). The minimum is attained exactly at the dilate $L_f$ of volume $\omega_n$ of the polar projection body of $f$, and $h(\Gamma_pL_f,\cdot)=H_f$. (Taking the minimum over all origin-symmetric convex bodies of volume $\omega_n$, instead of the bodies $\Gamma_pL$, leads to the optimal Sobolev norms of Lutwak, Yang and Zhang \cite{LYZ2006}.) Roughly speaking, this variational structure turns an affine problem into an anisotropic one, to which Euclidean methods apply, and the $L_p$ Busemann--Petty centroid inequality $|\Gamma_pL|\geq|L|$ of Lutwak, Yang and Zhang \cite{LYZ2000} (see also \cite{CG2002}) measures what is \emph{gained}. This strategy was originally exploited by Haddad, Jim\'enez and Montenegro \cite{HJM2016} and Nguyen \cite{Nguyen2016}, who used it for affine Sobolev inequalities and affine P\'olya--Szeg\H{o} principles, respectively.

In this work, we have decided to use this strategy in the proofs of Theorems \ref{thm:iso} and \ref{thm:isocap} (it is also the main tool of Theorem \ref{thm:KJ}), rather than using directly the affine P\'olya--Szeg\H{o} principle \eqref{e:affinePZ}. The decision is not purely cosmetic; as mentioned above, if $u$ is an extremal, the affine functional coincides with the anisotropic functional with norm $H_u$, and the anisotropic isoperimetric inequality compares it with the Euclidean value at the cost of a power of $|\Gamma_pL_u|/|L_u|\geq1$. This allows us to obtain quantitative forms of the isoperimetric inequalities via a stability result for the $L_p$ Busemann--Petty centroid inequality. This is the purpose of the second part of the paper; we first obtain such a stability result by using B\"or\"oczky's stability version of the Orlicz--Petty projection inequality \cite{Boroczky2013}. Since it may be of independent interest, we include the statement below.

\begin{theorem}\label{thm:BPstab}
Let $p\geq1$. There exists $\gamma=\gamma(n,p)>0$ such that
\[
\frac{|\Gamma_pL|}{|L|}\geq1+\gamma\,\delta_{\mathrm{BM}}(L,B^n)^{840(n+p)}
\]
for every origin-symmetric convex body $L\subset\mathbb{R}^n$.
\end{theorem}
Here $\delta_\mathrm{BM}$ is the Banach--Mazur distance between origin-symmetric convex bodies $Q_1,Q_2\subset\mathbb{R}^n$,
\[
\delta_{\mathrm{BM}}(Q_1,Q_2)=\min\{\log\lambda\ :\ Q_1\subseteq AQ_2\subseteq\lambda Q_1,\ A\in\mathrm{GL}(n)\}.
\]
We would like to remark that no attempt to optimize the exponent has been performed. For $p=1$, Theorem \ref{thm:BPstab} is due to Nguyen \cite[Proposition 5.3]{Nguyen2016}; in the plane, a stability version in the Banach--Mazur distance was obtained by Ivaki \cite{Ivaki2016}. We extend Nguyen's argument to every $p\geq1$.

In our applications, the affine deficit dominates the Busemann--Petty deficit of $L_u$, which a stability version of the $L_p$ Busemann--Petty centroid inequality turns into a distance from $L_u$ to ellipsoids. Using the latter together with the Euclidean stability results \eqref{eq:quantFK} and \eqref{eq:quantisocap}, we provide quantitative forms of \eqref{eq:iso} and \eqref{eq:affisocap}. In addition, by further assuming convexity of the domain we can get a stability result of \eqref{eq:KJ}. Let
\[
\mathcal{F}_{\mathrm{aff}}(K)=\inf\Big\{\frac{|K\triangle E|}{|K|}\ :\ E\text{ an ellipsoid with }|E|=|K|\Big\}
\]
be the affine Fraenkel asymmetry of $K$.

\begin{theorem}\label{thm:stab}
Let $\beta=2$ if $p=2$ and $\beta=1680(n+p)$ if $p\neq2$.
\begin{enumerate}[label=\textup{(\roman*)}]
\item Let $p>1$ and $1\leq q<p^*$. There is $c=c(n,p,q)>0$ such that, for every bounded open set $\Omega\subset\mathbb{R}^n$,
\[
\Lambda^{\mathcal{A}}_{p,q}(\Omega)\geq\big(1+c\,\mathcal{F}_{\mathrm{aff}}(\Omega)^{\beta}\big)\Lambda_{p,q}(\Omega^\star).
\]
In particular, $T^{\mathcal{A}}_p(\Omega)\leq\big(1+c\,\mathcal{F}_{\mathrm{aff}}(\Omega)^{\beta}\big)^{-1/(p-1)}T_p(\Omega^\star)$ when $q=1$.
\item Let $1<p<n$. There is $c=c(n,p)>0$ such that, for every compact set $K\subset\mathbb{R}^n$ with $|K|>0$,
\[
\operatorname{cap}^{\mathcal{A}}_p(K)\geq\big(1+c\,\mathcal{F}_{\mathrm{aff}}(K)^{\beta}\big)\operatorname{cap}_p(K^\star).
\]

\item Let $p>1$ and $1<q<p^*$. For every $\eta>0$ there is $\varepsilon=\varepsilon(n,p,q,\eta)>0$ with the following property. If $\Omega\subset\mathbb{R}^n$ is a bounded open convex set such that
\[
T^{\mathcal{A}}_p(\Omega)^{(p-1)\sigma}\,\Lambda^{\mathcal{A}}_{p,q}(\Omega)\leq(1+\varepsilon)\,T_p(B^n)^{(p-1)\sigma}\,\Lambda_{p,q}(B^n),
\]
then there is an open ellipsoid $E$, with center $z$, such that $E\subseteq\Omega\subseteq z+e^{\eta}(E-z)$; in particular, $\mathcal{F}_{\mathrm{aff}}(\Omega)\leq e^{n\eta}-1$.
\end{enumerate}
\end{theorem}

As pointed out above, the situation of Kohler-Jobin's inequality is different because no quantitative form of \eqref{eq:KJEuclid} seems to be available. By Proposition \ref{prop:bodystab}, the deficit in \eqref{eq:KJ} still controls $\delta_{\mathrm{BM}}(\Pi^*_pu,B^n)$, and Proposition \ref{prop:reduction} shows that it bounds, up to a power, the Euclidean Kohler-Jobin deficit of a linear image of $\Omega$. Consequently, any quantitative form of \eqref{eq:KJEuclid} valid on a class of sets invariant under linear maps would yield a quantitative form of \eqref{eq:KJ} on the same class, with the distance to balls replaced by the distance to ellipsoids and the exponent multiplied by $840(n+p)$, or unchanged if $p=2$. We will show that such a class cannot contain all open sets (Remark \ref{rem:dust}). Part (iii) is therefore a qualitative substitute on convex sets, and a quantitative form of \eqref{eq:KJEuclid} on convex sets would give it a rate.

We finish by pointing out that for $p=2$ the anisotropic norm is always Euclidean up to a linear map, the factor $|\Gamma_2L_u|/|L_u|$ equals $1$, and the affine problems  essentially reduce to Euclidean ones (see Section \ref{sec:p2}). Moreover, in this case, the affine energy of a function is the least Dirichlet energy of its images under $\mathrm{SL}(n)$ \cite{HuangLi2016,SchindlerTintarev2018}. Consequently $\operatorname{cap}^{\mathcal{A}}_2(K)$, $\Lambda^{\mathcal{A}}_{2,q}(\Omega)$ and $T^{\mathcal{A}}_2(\Omega)$ are extremal values of their Euclidean counterparts over the images of the set (Proposition \ref{prop:p2}), and for $p=2$ the equality statements in Theorems \ref{thm:iso}, \ref{thm:isocap} and \ref{thm:KJ} follow from their Euclidean counterparts applied to a suitable linear image of the set. Therefore, the substance of our main theorems lies in the case $p\neq2$.

This paper is organized as follows. In Section \ref{sec:prelim} we collect the background on convex bodies, Sobolev functions, capacities and the affine energy, including the anisotropic structure of $\mathcal{E}_p$, anisotropic functionals and the case $p=2$. In Section \ref{sec:existence} we prove that $\Lambda^{\mathcal{A}}_{p,q}(\Omega)$ is attained for all $p>1$ and $1\leq q<p^*$, and we prove Theorem \ref{thm:potential}. Section \ref{sec:iso} contains the proofs of Theorems \ref{thm:iso} and \ref{thm:isocap} and of Corollary \ref{cor:SV}, and the comparison of the affine and Euclidean capacities (Proposition \ref{prop:compare}); Section \ref{sec:KJ} contains the proof of Theorem \ref{thm:KJ}, whereas Theorem \ref{thm:BPstab} is proved in Section \ref{sec:BPstab}. Finally, Section \ref{sec:stability} contains the bounds in terms of the polar projection body of an extremal (Proposition \ref{prop:bodystab}), the proof of Theorem \ref{thm:stab}, and the reduction of the affine Kohler-Jobin deficit to the Euclidean one (Proposition \ref{prop:reduction}).

\section{Background and preliminaries}\label{sec:prelim}

We work in the $n$-dimensional Euclidean space $\mathbb{R}^n$, $n\geq2$, endowed with the standard inner product $\langle\cdot,\cdot\rangle$ and the Euclidean norm $|\cdot|$. We denote by $B^n$ the Euclidean unit ball, by $S^{n-1}$ its boundary and by $\omega_n$ the volume of $B^n$; the Lebesgue measure of a measurable set $A$ is denoted by $|A|$, and $d\xi$ stands for integration with respect to the $(n-1)$-dimensional Hausdorff measure on $S^{n-1}$. Throughout, $1<p<\infty$, $p'=p/(p-1)$, and $p^*=np/(n-p)$ if $p<n$, $p^*=\infty$ if $p\geq n$. In statements involving $D^{1,p}(\mathbb{R}^n)$ or capacities of compact sets we assume $1<p<n$. For $z\in\mathbb{R}^n$ and $r>0$, $B_r(z)$ is the open ball of radius $r$ centered at $z$, and $B_r=B_r(0)$.

\subsection{Convex bodies}\label{sec:convex}
We refer to the books by Gardner \cite{Gardner2006} and Schneider \cite{Schneider2014} for background. A convex body is a compact convex set with non-empty interior, and $\mathcal{K}^n_o$ denotes the class of convex bodies containing the origin in their interiors. For $Q\in\mathcal{K}^n_o$, its support function, its gauge and its polar body are
\[
h_Q(x)=\max_{y\in Q}\langle x,y\rangle,\qquad \|x\|_Q=\min\{\lambda\geq0:x\in\lambda Q\},\qquad Q^\circ=\{x\in\mathbb{R}^n:h_Q(x)\leq1\},
\]
so that $\|\cdot\|_{Q^\circ}=h_Q$; if $Q$ is origin-symmetric, $h_Q$ and $\|\cdot\|_Q$ are norms. A star body is a compact set $L$, star-shaped with respect to the origin, whose radial function $\rho_L(\xi)=\max\{\lambda\geq0:\lambda\xi\in L\}$ is positive and continuous on $S^{n-1}$; in polar coordinates, $|L|=\frac1n\int_{S^{n-1}}\rho_L(\xi)^n\,d\xi$. An ellipsoid is the image of $B^n$, or of its closure, under an invertible affine map; it is origin-symmetric if the map is linear.

For a convex body $Q$, $S_Q$ denotes its surface area measure on $S^{n-1}$. The mixed volume of two convex bodies $Q$ and $Z$ satisfies $V(Q[n-1],Z)=\frac1n\int_{S^{n-1}}h_Z\,dS_Q$; it is increasing in each argument with respect to inclusion and positively homogeneous of degree $n-1$ in $Q$ \cite[Section 5.1]{Schneider2014}. Following Lutwak \cite{Lutwak1993}, the $L_p$ mixed volume of $M,N\in\mathcal{K}^n_o$ is
\[
V_p(M,N)=\frac1n\int_{S^{n-1}}h_N^p\,h_M^{1-p}\,dS_M,
\]
so that $V_p(M,M)=|M|$. For star bodies $L$ and $Q$, the dual mixed volume $\widetilde{V}_{-p}(L,Q)=\frac1n\int_{S^{n-1}}\rho_L^{n+p}\rho_Q^{-p}\,d\xi$ satisfies, by H\"older's inequality,
\begin{equation}\label{eq:dualmixed}
\widetilde{V}_{-p}(L,Q)\geq|L|^{\frac{n+p}{n}}|Q|^{-\frac{p}{n}},
\end{equation}
with equality if and only if $L$ and $Q$ are dilates.

Let
\begin{equation}\label{eq:anp}
a_{n,p}=\frac{1}{n\omega_n}\int_{S^{n-1}}|\langle\xi,v\rangle|^p\,d\xi=\frac{\Gamma(\frac n2)\,\Gamma(\frac{p+1}2)}{\sqrt{\pi}\,\Gamma(\frac{n+p}2)},\qquad v\in S^{n-1}.
\end{equation}
The $L_p$ centroid body $\Gamma_pL$ of a star body $L$ and the $L_p$ projection body $\Pi_pM$ of $M\in\mathcal{K}^n_o$ are the origin-symmetric convex bodies given by
\begin{equation}\label{eq:GammaPi}
\begin{aligned}
h(\Gamma_pL,x)^p&=\frac{n+p}{n\,a_{n,p}\,|L|}\int_L|\langle x,y\rangle|^p\,dy,\\
h(\Pi_pM,x)^p&=\frac{1}{n\omega_n a_{n,p}}\int_{S^{n-1}}|\langle x,v\rangle|^p\,h_M(v)^{1-p}\,dS_M(v),
\end{aligned}
\end{equation}
where we write $h(Q,\cdot)=h_Q$. They are normalized so that $\Gamma_pB^n=\Pi_pB^n=B^n$, and coincide with the definitions of Lutwak, Yang and Zhang \cite{LYZ2000}. We write $\Pi^*_pM=(\Pi_pM)^\circ$. For $A\in\mathrm{GL}(n)$ one has $\Gamma_p(AL)=A\,\Gamma_pL$ and $\Pi^*_p(AM)=|\det A|^{-1/p}A\,\Pi^*_pM$ \cite{LYZ2000}. The $L_p$ Busemann--Petty centroid inequality \cite{LYZ2000} (see also \cite{CG2002} for a different proof) states that
\begin{equation}\label{eq:BP}
|\Gamma_pL|\geq|L|\quad\text{for every star body }L,
\end{equation}
with equality if and only if $L$ is an origin-symmetric ellipsoid.

For origin-symmetric convex bodies $Q_1$ and $Q_2$, the Banach--Mazur distance $\delta_{\mathrm{BM}}(Q_1,Q_2)$ defined in the introduction is invariant under linear images of either body and satisfies the triangle inequality. By John's theorem \cite{John1948}, every origin-symmetric convex body $Q$ contains an origin-symmetric ellipsoid $E$ with $Q\subseteq\sqrt n\,E$; hence $\delta_{\mathrm{BM}}(Q,B^n)\leq\frac12\log n$ and $\delta_{\mathrm{BM}}(Q_1,Q_2)\leq\log n$.

\subsection{Sobolev functions, rearrangements and capacities}\label{sec:functions}
Let $\Omega\subseteq\mathbb{R}^n$ be open. We denote by $W^{1,p}_0(\Omega)$ the closure of $C^\infty_c(\Omega)$ in $W^{1,p}(\mathbb{R}^n)$, with the norm $\|u\|_{W^{1,p}}=\big(\int_{\mathbb{R}^n}(|\nabla u|^p+|u|^p)\,dx\big)^{1/p}$, which is the one used in the definition of $C_p$ below; by the Poincar\'e inequality, $\|\nabla\cdot\|_p$ is an equivalent norm on $W^{1,p}_0(\Omega)$ when $\Omega$ is bounded, for every $p\geq1$; its elements are extended by zero outside $\Omega$, and $\|\cdot\|_q$ denotes the norm of $L^q(\mathbb{R}^n)$. For a measurable set $A$ of finite measure, $A^\star$ is the open ball centered at the origin with $|A^\star|=|A|$; for a compact set $K$, we let $K^\star$ be instead the closed ball centered at the origin with $|K^\star|=|K|$. For a measurable function $f$ whose level sets $\{|f|>t\}$, $t>0$, have finite measure, $f^\star=\int_0^\infty\chi_{\{|f|>t\}^\star}\,dt$ is the symmetric decreasing rearrangement of $|f|$; it is radially symmetric, radially nonincreasing and equimeasurable with $|f|$, so that $\|f^\star\|_q=\|f\|_q$. For $f\in W^{1,p}_0(\Omega)$, the P\'olya--Szeg\H{o} principle \cite{PolyaSzego1951,BrothersZiemer1988} gives $f^\star\in W^{1,p}_0(\Omega^\star)$ and $\|\nabla f^\star\|_p\leq\|\nabla f\|_p$, which implies \eqref{eq:FK}. By the Rellich--Kondrachov theorem, $\Lambda_{p,q}(\Omega)$ is attained when $\Omega$ is bounded.

For $t>0$ one has $\Lambda_{p,q}(t\Omega)=t^{n-p-np/q}\Lambda_{p,q}(\Omega)$ and $T_p(t\Omega)=t^{n+p'}T_p(\Omega)$, which shows that the left-hand side of \eqref{eq:KJEuclid} is invariant under dilations. The solution of $-\Delta_pu=1$ in the ball $B_R$ of radius $R$ with $u=0$ on $\partial B_R$ is $u(x)=\frac{p-1}{p}\,n^{-1/(p-1)}(R^{p'}-|x|^{p'})$ \cite[Remark 2.3]{Brasco2014}, whence
\begin{equation}\label{eq:TBall}
T_p(B_R)=\int_{B_R}u\,dx=\frac{\omega_n\,R^{n+p'}}{n^{\frac1{p-1}}\,(n+p')}.
\end{equation}

The Sobolev $p$-capacity of a set $E\subseteq\mathbb{R}^n$ is
\[
C_p(E)=\inf\Big\{\int_{\mathbb{R}^n}\big(|\nabla\varphi|^p+|\varphi|^p\big)\,dx\ :\ \varphi\in W^{1,p}(\mathbb{R}^n),\ \varphi\geq1\text{ a.e. in a neighborhood of }E\Big\}.
\]
Sets of zero Sobolev $p$-capacity have zero measure, and for $p>n$ only the empty set has zero Sobolev $p$-capacity, since by Morrey's inequality every point has positive Sobolev $p$-capacity; a property holds quasi-everywhere (q.e.) if it holds outside a set of zero Sobolev $p$-capacity.  A function $f$ is quasicontinuous if for every $\varepsilon>0$ there is an open set $G$ with $C_p(G)<\varepsilon$ such that the restriction of $f$ to $\mathbb{R}^n\setminus G$ is continuous. Every $u\in W^{1,p}(\mathbb{R}^n)$ has a quasicontinuous representative $\tilde u$, unique q.e., and two quasicontinuous functions that coincide a.e. in an open set coincide q.e. in it. For $p>n$, $\tilde u$ is the continuous representative of $u$. We shall use the following characterization, valid for every open set $\Omega$ and every $u\in W^{1,p}(\mathbb{R}^n)$ (see \cite[Chapter 4]{HKM2006} and \cite[Section 9.1]{AdamsHedberg1996}):
\begin{equation}\label{eq:W0cap}
u\in W^{1,p}_0(\Omega)\iff\tilde u=0\ \text{q.e. in }\mathbb{R}^n\setminus\Omega.
\end{equation}
Consequently, if $\Omega\subseteq\Omega'$ are open and $C_p(\Omega'\setminus\Omega)=0$, then $W^{1,p}_0(\Omega)=W^{1,p}_0(\Omega')$, and all the quantities considered in this paper coincide for $\Omega$ and $\Omega'$. In particular, if $\Omega$ is an ellipsoid up to a set of zero Sobolev $p$-capacity, then $\Omega$ can be replaced by the corresponding open ellipsoid.

Now let $1<p<n$. The homogeneous Sobolev space $D^{1,p}(\mathbb{R}^n)$ is the completion of $C^\infty_c(\mathbb{R}^n)$ with respect to $\|\nabla\cdot\|_p$. By the Sobolev inequality it is identified with $\{u\in L^{p^*}(\mathbb{R}^n):\nabla u\in L^p(\mathbb{R}^n)\}$; in particular, it contains $W^{1,p}(\mathbb{R}^n)$ and embeds continuously into $L^{p^*}(\mathbb{R}^n)$. Every $u\in D^{1,p}(\mathbb{R}^n)$ has a quasicontinuous representative $\tilde u$, unique q.e., and if $u_j\to u$ in $D^{1,p}(\mathbb{R}^n)$, a subsequence satisfies $\tilde u_j\to\tilde u$ q.e. \cite[Chapter 4]{HKM2006}. For a compact set $K$ let
\[
\mathcal{W}(K)=\{u\in D^{1,p}(\mathbb{R}^n):\ \tilde u\geq1\ \text{q.e. on }K\}.
\]
This class is convex and closed, hence weakly closed, and it is stable under the truncation $u\mapsto\min\{\max\{u,0\},1\}$. Every $u\in\mathcal{W}(K)$ is the limit in $D^{1,p}(\mathbb{R}^n)$ of functions $\varphi_j\in C^\infty_c(\mathbb{R}^n)$ with $\varphi_j\geq1$ on $K$ \cite[Chapters 2 and 4]{HKM2006}. Consequently, the infimum in \eqref{eq:cap} may be taken over $\mathcal{W}(K)$, and, since $\mathcal{E}_p$ is continuous on $D^{1,p}(\mathbb{R}^n)\setminus\{0\}$ (Proposition \ref{prop:energy}(iii)), so may the infimum in \eqref{eq:affcap}. We shall also use that if $U\subseteq\mathbb{R}^n$ is open and $w\in D^{1,p}(\mathbb{R}^n)$ satisfies $\tilde w=0$ q.e. on $\mathbb{R}^n\setminus U$, then $w$ is the limit in $D^{1,p}(\mathbb{R}^n)$ of functions in $C^\infty_c(U)$ \cite[Chapter 4]{HKM2006}.

Two capacities appear in this paper, with different roles. The variational capacity $\operatorname{cap}_p$ of \eqref{eq:cap}, for $1<p<n$, is the quantity studied in the isocapacitary problems. The Sobolev capacity $C_p$, defined for every $p>1$, only measures exceptional sets. It has to be used for this purpose when $p\geq n$, because then the variational capacity of every compact subset of $\mathbb{R}^n$ vanishes. For $1<p<n$ the two capacities are comparable on bounded sets, with constants depending only on $n$, $p$ and the radius of a ball containing the set; in particular, they have the same null sets \cite[Chapter 2]{HKM2006}.

Capacities scale as $\operatorname{cap}_p(\lambda K)=\lambda^{n-p}\operatorname{cap}_p(K)$, $\lambda>0$, and the same holds for $\operatorname{cap}^{\mathcal{A}}_p$. The capacitary potential of the closed ball $\overline{B}_R$ of radius $R$ centered at the origin is $v_R(x)=\min\{1,(|x|/R)^{(p-n)/(p-1)}\}$, and
\begin{equation}\label{eq:capball}
\operatorname{cap}_p(\overline{B}_R)=n\omega_n\Big(\frac{n-p}{p-1}\Big)^{p-1}R^{n-p}.
\end{equation}
For an open set, the capacity is defined by \eqref{eq:cap} with $K$ replaced by the open set. If $K$ is compact and $K_t=\{x\in\mathbb{R}^n:\operatorname{dist}(x,K)<t\}$ is its open $t$-neighborhood, then $\operatorname{cap}_p(K_t)\to\operatorname{cap}_p(K)$ as $t\to0$: if $\varphi\in C^\infty_c(\mathbb{R}^n)$ satisfies $\varphi\geq1$ on $K$ and $\eta>0$, then $(1+\eta)\varphi\geq1$ on the open set $\{\varphi>(1+\eta)^{-1}\}$, which contains $K_t$ for $t$ small. This allows us to state the sharp quantitative isocapacitary inequality for compact sets.

\begin{theorem}[{\cite{DPMM2021,Mukoseeva2023}}]\label{thm:DPMM}
There exists $c_{\mathrm{cap}}=c_{\mathrm{cap}}(n,p)>0$ such that
\[
\operatorname{cap}_p(K)\geq\big(1+c_{\mathrm{cap}}\,\mathcal{F}(K)^2\big)\operatorname{cap}_p(K^\star)
\]
for every compact set $K\subset\mathbb{R}^n$ with $|K|>0$.
\end{theorem}

\begin{proof}
For open sets of finite measure this is \cite[Theorem 1.4]{Mukoseeva2023} (see \cite{DPMM2021} for $p=2$), written in scale-invariant form by means of the scaling of the capacity. Let $K$ be compact with $|K|>0$ and apply it to the open neighborhoods $K_t$. As $t\to0$, $\operatorname{cap}_p(K_t)\to\operatorname{cap}_p(K)$, $|K_t|\to|K|$ and $\operatorname{cap}_p(K_t^\star)\to\operatorname{cap}_p(K^\star)$, where $K_t^\star$ is the ball of volume $|K_t|$. If $B_t$ is a ball with $|B_t|=|K_t|$ and $|K_t\triangle B_t|\leq(\mathcal{F}(K_t)+t)|K_t|$, and $B'_t$ is the concentric ball with $|B'_t|=|K|$, then $|K\triangle B'_t|\leq(\mathcal{F}(K_t)+t)|K_t|+2(|K_t|-|K|)$, so that $\mathcal{F}(K)\leq\liminf_{t\to0}\mathcal{F}(K_t)$. The inequality for $K$ follows by letting $t\to0$.
\end{proof}

\subsection{The affine energy}\label{sec:energy}
Let
\[
c_{n,p}=(n\omega_n)^{1/n}\,a_{n,p}^{-1/p}.
\]
For $f\in W^{1,p}(\mathbb{R}^n)$, or $f\in D^{1,p}(\mathbb{R}^n)$ if $1<p<n$, and $\xi\in\mathbb{R}^n$, we write $\nabla_\xi f=\langle\nabla f,\xi\rangle$, and we call
\[
\Pi^*_pf=\{\xi\in\mathbb{R}^n:\ \|\nabla_\xi f\|_p\leq1\}
\]
the polar projection body of $f$. By Fubini's theorem,
\begin{equation}\label{eq:average}
\frac{1}{n\omega_n}\int_{S^{n-1}}\|\nabla_\xi f\|_p^p\,d\xi=a_{n,p}\,\|\nabla f\|_p^p.
\end{equation}
By the inequality between the power means of orders $-n$ and $p$ with respect to the normalized measure on $S^{n-1}$, and by \eqref{eq:average},
\begin{equation}\label{eq:domination}
\mathcal{E}_p(f)\leq c_{n,p}\,(n\omega_n)^{-1/n}\Big(\frac{1}{n\omega_n}\int_{S^{n-1}}\|\nabla_\xi f\|_p^p\,d\xi\Big)^{1/p}=\|\nabla f\|_p,
\end{equation}
with equality if and only if $\xi\mapsto\|\nabla_\xi f\|_p$ is constant on $S^{n-1}$; in particular, $\mathcal{E}_p(f)=\|\nabla f\|_p$ when $f$ is radial. This is the right inequality in \eqref{eq:affineSobolev}. We collect the properties of $\mathcal{E}_p$ that we need.

\begin{proposition}\label{prop:energy}
Let $p>1$, let $\Omega\subset\mathbb{R}^n$ be a bounded open set, and let $f\in W^{1,p}(\mathbb{R}^n)$, or $f\in D^{1,p}(\mathbb{R}^n)$ if $1<p<n$.
\begin{enumerate}[label=\textup{(\roman*)}]
\item For $A\in\mathrm{SL}(n)$, $x_0\in\mathbb{R}^n$ and $t>0$, one has $\mathcal{E}_p(|f|)=\mathcal{E}_p(f\circ A)=\mathcal{E}_p(f(\cdot+x_0))=\mathcal{E}_p(f)$ and $\mathcal{E}_p(f(\cdot/t))=t^{(n-p)/p}\,\mathcal{E}_p(f)$.
\item There are $t_p>0$, depending only on $p$, and $C=C(n,p,\Omega)>0$ such that, for every $g\in W^{1,p}_0(\Omega)$ and $\xi\in S^{n-1}$,
\[
\|\nabla_\xi g\|_p\geq t_p\,w(\Omega,\xi)^{-1}\|g\|_p\qquad\text{and}\qquad\mathcal{E}_p(g)\geq C\,\|g\|_p^{\frac{n-1}{n}}\|\nabla g\|_p^{\frac1n},
\]
where $w(\Omega,\xi)$ is the width of $\Omega$ in the direction $\xi$.
\item $\mathcal{E}_p$ is continuous for the norm topology on $W^{1,p}_0(\Omega)\setminus\{0\}$ and, if $1<p<n$, on $D^{1,p}(\mathbb{R}^n)\setminus\{0\}$, and it is sequentially weakly lower semicontinuous on $W^{1,p}_0(\Omega)$.
\item If $1<p<n$, then $S_{n,p}\|f\|_{p^*}\leq\mathcal{E}_p(f)$.
\end{enumerate}
\end{proposition}

Part (i) follows directly from \eqref{eq:Ep} (see \cite{LYZ2002}). Part (ii) is \cite[Lemma 1 and Theorem 9]{HJM2021}, stated there for $C^1$ functions supported in $\Omega$ and extended to $W^{1,p}_0(\Omega)$ by density. For the continuity in (iii), note that $\big|\|\nabla_\xi f\|_p-\|\nabla_\xi g\|_p\big|\leq\|\nabla f-\nabla g\|_p$ for every $\xi\in S^{n-1}$, so the directional energies converge uniformly on $S^{n-1}$ along a convergent sequence; if $f_j\to f\neq0$ in $D^{1,p}(\mathbb{R}^n)$, then $\|\nabla_\xi f_j\|_p\to\|\nabla_\xi f\|_p$ uniformly in $\xi\in S^{n-1}$, and therefore $\mathcal{E}_p(f_j)\to\mathcal{E}_p(f)$. The weak lower semicontinuity in (iii) is proved in \cite[proof of Theorem 3]{HJM2021} (see also \cite[Theorem 2.1]{LM2024ACV}). Part (iv) is the affine Sobolev inequality of \cite{LYZ2002}, which extends from $W^{1,p}(\mathbb{R}^n)$ to $D^{1,p}(\mathbb{R}^n)$ by density.

We recall that if  $f\in W^{1,p}_0(\Omega)\setminus\{0\}$ for a bounded open set $\Omega$, or $f\in D^{1,p}(\mathbb{R}^n)\setminus\{0\}$ with $1<p<n$, the function $\xi\mapsto\|\nabla_\xi f\|_p$ is a norm on $\mathbb{R}^n$. It is continuous and subadditive by Minkowski's inequality, and it does not vanish on $S^{n-1}$: in the first case by (ii), and in the second because a function in $D^{1,p}(\mathbb{R}^n)$ that is constant along all lines of a fixed direction vanishes. Hence $\Pi^*_pf$ is an origin-symmetric convex body, $\mathcal{E}_p(f)=c_{n,p}(n|\Pi^*_pf|)^{-1/n}>0$, and we may define
\begin{equation}\label{eq:Lf}
L_f=\Big(\frac{\omega_n}{|\Pi^*_pf|}\Big)^{1/n}\,\Pi^*_pf,\qquad |L_f|=\omega_n.
\end{equation}

The next lemma represents $\mathcal{E}_p^p$ as a minimum of anisotropic Dirichlet energies. For $L=L_f$ the identity in \eqref{eq:frozen} is contained in \cite[Lemmas 3 and 4]{HJM2016}.

\begin{lemma}\label{lem:frozen}
Let $f\in W^{1,p}_0(\Omega)\setminus\{0\}$ for a bounded open set $\Omega\subset\mathbb{R}^n$, or $f\in D^{1,p}(\mathbb{R}^n)\setminus\{0\}$ with $1<p<n$. For every star body $L$ with $|L|=\omega_n$,
\begin{equation}\label{eq:frozen}
\mathcal{E}_p(f)^p\leq\int_{\mathbb{R}^n}h\big(\Gamma_pL,\nabla f(x)\big)^p\,dx,
\end{equation}
with equality if and only if $L=L_f$. Moreover, $h(\Gamma_pL_f,\cdot)=H_f$, the norm in \eqref{eq:affineLaplacian}.
\end{lemma}

\begin{proof}
By \eqref{eq:GammaPi} and polar coordinates, $h(\Gamma_pL,\zeta)^p=\frac{1}{n\,a_{n,p}\,|L|}\int_{S^{n-1}}\rho_L(\xi)^{n+p}|\langle\zeta,\xi\rangle|^p\,d\xi$ for $\zeta\in\mathbb{R}^n$. Hence, by Fubini's theorem, $|L|=\omega_n$ and \eqref{eq:dualmixed},
\begin{align}
    \int_{\mathbb{R}^n}h(\Gamma_pL,\nabla f)^p\,dx&=\frac{1}{n\omega_n a_{n,p}}\int_{S^{n-1}}\rho_L(\xi)^{n+p}\,\|\nabla_\xi f\|_p^p\,d\xi\\
    &= \frac{1}{\omega_n a_{n,p}}\widetilde{V}_{-p}(L,\Pi^*_pf)\\
    &\geq a_{n,p}^{-1}\, \omega_n^{\frac{p}{n}}|\Pi^*_pf|^{-\frac{p}{n}} = \mathcal{E}_p(f)^p.
\end{align}
Equality holds if and only if $L$ and $\Pi_p^* f$ are dilates, that is, if and only if $L = L_f$.
\end{proof}

In particular, as observed in \cite{HJM2021}, $\Delta^{\mathcal{A}}_pu$ is the anisotropic $p$-Laplacian of $u$ with respect to the norm $h(\Gamma_pL_u,\cdot)$. Moreover, the anisotropic norms $H_f$ have the following regularity.

\begin{lemma}\label{lem:admissible}
Let $p>1$ and let $L$ be a star body. Then the norm $h(\Gamma_pL,\cdot)$ is strictly convex and of class $C^1$ on $\mathbb{R}^n\setminus\{0\}$.
\end{lemma}

\begin{proof}
By \eqref{eq:GammaPi}, $h(\Gamma_pL,x)$ is a constant multiple of the $L^p(L)$ norm of the linear function $y\mapsto\langle x,y\rangle$. If $h(\Gamma_pL,x+x')=h(\Gamma_pL,x)+h(\Gamma_pL,x')$ with $x,x'\neq0$, the equality case of Minkowski's inequality, $1<p<\infty$, shows that $\langle x,\cdot\rangle$ and $\langle x',\cdot\rangle$ are positively proportional on $L$, which has non-empty interior; hence $x'$ is a positive multiple of $x$, and $(\Gamma_pL)^\circ$ is strictly convex. Differentiation under the integral sign, legitimate because $p>1$ and $L$ is bounded, shows that $h(\Gamma_pL,\cdot)^p$ is of class $C^1$ on $\mathbb{R}^n$, and hence $h(\Gamma_pL,\cdot)$ is of class $C^1$ on $\mathbb{R}^n\setminus\{0\}$. Since $\langle\nabla h(\Gamma_pL,x),x\rangle=h(\Gamma_pL,x)=1$ on $\partial(\Gamma_pL)^\circ$, the gradient does not vanish there, and $(\Gamma_pL)^\circ$ has $C^1$ boundary.
\end{proof}

Let $\Omega\subset\mathbb{R}^n$ be a bounded open set, $A\in\mathrm{SL}(n)$, $x_0\in\mathbb{R}^n$ and $t>0$. Proposition \ref{prop:energy}(i) and \eqref{eq:domination} give
\begin{equation}\label{eq:invariance}
\Lambda^{\mathcal{A}}_{p,q}(A\Omega+x_0)=\Lambda^{\mathcal{A}}_{p,q}(\Omega),\quad\Lambda^{\mathcal{A}}_{p,q}(t\Omega)=t^{n-p-\frac{np}{q}}\,\Lambda^{\mathcal{A}}_{p,q}(\Omega),\quad\Lambda^{\mathcal{A}}_{p,q}(\Omega)\leq\Lambda_{p,q}(\Omega),
\end{equation}
and consequently $T^{\mathcal{A}}_p(t\Omega)=t^{n+p'}\,T^{\mathcal{A}}_p(\Omega)$ and $T^{\mathcal{A}}_p(\Omega)\geq T_p(\Omega)$. Similarly, if $1<p<n$ and $K$ is compact, then $\operatorname{cap}^{\mathcal{A}}_p(AK+x_0)=\operatorname{cap}^{\mathcal{A}}_p(K)$, $\operatorname{cap}^{\mathcal{A}}_p(tK)=t^{n-p}\operatorname{cap}^{\mathcal{A}}_p(K)$ and $\operatorname{cap}^{\mathcal{A}}_p(K)\leq\operatorname{cap}_p(K)$, because $u\in\mathcal{W}(K)$ if and only if $u(A^{-1}(\cdot-x_0))\in\mathcal{W}(AK+x_0)$.

\subsection{Anisotropic functionals}\label{sec:aniso}
For an origin-symmetric convex body $Q$, an open set $\Omega$ of finite measure and $1\leq q<p^*$, we consider the anisotropic constants
\[
\Lambda^{Q}_{p,q}(\Omega)=\inf\Big\{\int_\Omega h_Q(\nabla v)^p\,dx\ :\ v\in W^{1,p}_0(\Omega),\ \|v\|_q=1\Big\},\qquad T^{Q}_p(\Omega)=\Lambda^{Q}_{p,1}(\Omega)^{-\frac1{p-1}},
\]
and, if $1<p<n$, the anisotropic capacity of a compact set $K$,
\[
\operatorname{cap}^{Q}_p(K)=\inf\Big\{\int_{\mathbb{R}^n}h_Q(\nabla u)^p\,dx\ :\ u\in\mathcal{W}(K)\Big\},
\]
where, as for $\operatorname{cap}_p$, the infimum may be taken over functions $\varphi\in C^\infty_c(\mathbb{R}^n)$ with $\varphi\geq1$ on $K$. Thus $\Lambda^{B^n}_{p,q}=\Lambda_{p,q}$, $T^{B^n}_p=T_p$ and $\operatorname{cap}^{B^n}_p=\operatorname{cap}_p$. The unit ball of the norm $h_Q$ is $Q^\circ$, and the extremal sets of the anisotropic problems below are the dilates of $Q$. Exchanging the infima over functions and over star bodies in Lemma \ref{lem:frozen}, we obtain, for every bounded open set $\Omega$, every $1\leq q<p^*$ and, if $1<p<n$, every compact set $K$,
\begin{equation}\label{eq:envelope}
\begin{gathered}
\Lambda^{\mathcal{A}}_{p,q}(\Omega)=\inf_L\Lambda^{\Gamma_pL}_{p,q}(\Omega),\qquad T^{\mathcal{A}}_p(\Omega)=\sup_LT^{\Gamma_pL}_p(\Omega),\\
\operatorname{cap}^{\mathcal{A}}_p(K)=\inf_L\operatorname{cap}^{\Gamma_pL}_p(K),
\end{gathered}
\end{equation}
where $L$ ranges over  star bodies of volume $\omega_n$.

\begin{lemma}\label{lem:Wulff}
Let $Q$ be an origin-symmetric convex body.
\begin{enumerate}[label=\textup{(\roman*)}]
\item For $1\leq q<p^*$,
\[
\Lambda^{Q}_{p,q}(Q)=\Big(\frac{|Q|}{\omega_n}\Big)^{1-\frac pq}\Lambda_{p,q}(B^n)\qquad\text{and}\qquad T^{Q}_p(Q)=\frac{|Q|}{\omega_n}\,T_p(B^n).
\]
Consequently, for $1<q<p^*$,
\begin{equation}\label{eq:Wulff}
T^{Q}_p(Q)^{(p-1)\sigma}\,\Lambda^{Q}_{p,q}(Q)=\Big(\frac{|Q|}{\omega_n}\Big)^{\theta}\kappa,\qquad\theta=\frac{p^2(q-1)}{q\,(n(p-1)+p)}>0,
\end{equation}
where $\kappa=T_p(B^n)^{(p-1)\sigma}\,\Lambda_{p,q}(B^n)$.
\item If $1<p<n$ and $K$ is compact, then $\operatorname{cap}^{Q}_p(K)\geq(|Q|/\omega_n)^{p/n}\operatorname{cap}_p(K^\star)$.
\item If $Q=AB^n$ with $A\in\mathrm{GL}(n)$ and $\det A>0$, then, for every open set $\Omega$ of finite measure, every $1\leq q<p^*$ and, if $1<p<n$, every compact set $K$,
\[
\begin{gathered}
\Lambda^{Q}_{p,q}(\Omega)=(\det A)^{1-\frac pq}\,\Lambda_{p,q}(A^{-1}\Omega),\qquad T^{Q}_p(\Omega)=(\det A)\,T_p(A^{-1}\Omega),\\
\operatorname{cap}^{Q}_p(K)=(\det A)\operatorname{cap}_p(A^{-1}K).
\end{gathered}
\]
In particular, if $A=tA_0$ with $t>0$ and $A_0\in\mathrm{SL}(n)$, then $\operatorname{cap}^{Q}_p(K)=t^p\operatorname{cap}_p(A_0^{-1}K)$.
\item If $\Omega$ is a bounded open set and $1\leq q<p^*$, then
\[
\Lambda^{Q}_{p,q}(\Omega)\geq\Big(\frac{|Q|}{\omega_n}\Big)^{\frac pn}\Lambda_{p,q}(\Omega^\star)\qquad\text{and}\qquad T^{Q}_p(\Omega)\leq\Big(\frac{|Q|}{\omega_n}\Big)^{-\frac{p}{n(p-1)}}T_p(\Omega^\star).
\]
\end{enumerate}
\end{lemma}

\begin{proof}
(i) By the convex symmetrization of Alvino, Ferone, Lions and Trombetti \cite[Theorem 3.1]{AFLT1997}, for every nonnegative $v\in W^{1,p}_0(Q)$ the function equimeasurable with $v$ whose superlevel sets are dilates of $Q$ belongs to $W^{1,p}_0(Q)$ and has smaller or equal anisotropic energy $\int h_Q(\nabla\cdot)^p\,dx$. Hence the infimum defining $\Lambda^{Q}_{p,q}(Q)$ may be restricted to functions of the form $v(x)=g(\|x\|_Q)$, where $g$ is locally absolutely continuous on $(0,1]$ and $g(1)=0$. For such $v$, $h_Q(\nabla v(x))=|g'(\|x\|_Q)|$ for a.e. $x$, because $h_Q(\nabla\|x\|_Q)=1$ wherever the gauge is differentiable. Since $|\{\|x\|_Q<t\}|=t^n|Q|$, the coarea formula gives
\[
\int_Qh_Q(\nabla v)^p\,dx=n|Q|\int_0^1|g'(t)|^pt^{n-1}\,dt,\qquad\int_Q|v|^q\,dx=n|Q|\int_0^1|g(t)|^qt^{n-1}\,dt.
\]
The same identities hold for the Euclidean energy of $g(|x|)$ on $B^n$, with $|Q|$ replaced by $\omega_n$. Consequently $\Lambda^{Q}_{p,q}(Q)=(n|Q|)^{1-p/q}m$ and $\Lambda_{p,q}(B^n)=(n\omega_n)^{1-p/q}m$, where $m$ is the infimum of
\[
\frac{\int_0^1|g'(t)|^p\,t^{n-1}\,dt}{\big(\int_0^1|g(t)|^q\,t^{n-1}\,dt\big)^{p/q}}
\]
over all functions $g\not\equiv0$ that are locally absolutely continuous on $(0,1]$ with $g(1)=0$.
The case $q=1$ gives the second one (see also \cite[Lemma 6.2]{Brasco2014}). Finally, \eqref{eq:Wulff} follows from the identity $(p-1)\sigma+1-\frac pq=\theta$.

(ii) Let $\varphi\in C^\infty_c(\mathbb{R}^n)$ with $\varphi\geq1$ on $K$; replacing $\varphi$ by $\min\{\max\{\varphi,0\},1\}$, which does not increase the anisotropic energy, we may assume that $\varphi$ is Lipschitz, compactly supported and takes values in $[0,1]$. As above, its convex symmetrization $\varphi^{Q}$ satisfies $\int h_Q(\nabla\varphi^{Q})^p\leq\int h_Q(\nabla\varphi)^p$ by \cite[Theorem 3.1]{AFLT1997}. Moreover, $\varphi^{Q}\geq1$ on the dilate $sQ$ of $Q$ with $|sQ|=|K|$, because $|\{\varphi\geq1\}|\geq|K|$. As in (i), for a function $v(x)=g(\|x\|_Q)$ the coarea formula gives
\[
\int_{\mathbb{R}^n}h_Q(\nabla v)^p\,dx=\frac{|Q|}{\omega_n}\int_{\mathbb{R}^n}\big|\nabla\big(g(|x|)\big)\big|^p\,dx.
\]
Writing $\varphi^{Q}(x)=g(\|x\|_Q)$, the function $g(|x|)$ is admissible for $\overline{B}_s$, and therefore $\int h_Q(\nabla\varphi)^p\geq(|Q|/\omega_n)\operatorname{cap}_p(\overline{B}_s)$. Since $|\overline{B}_s|=\omega_n|K|/|Q|$, the scaling of the capacity gives
\[
\frac{|Q|}{\omega_n}\operatorname{cap}_p(\overline{B}_s)=\Big(\frac{|Q|}{\omega_n}\Big)^{p/n}\operatorname{cap}_p(K^\star),
\]
and (ii) follows by taking the infimum over $\varphi$.

(iii) We have $h_Q(\zeta)=|A^{\mathrm{T}}\zeta|$. If $v\in W^{1,p}_0(\Omega)$ and $w(y)=v(Ay)$, then $w\in W^{1,p}_0(A^{-1}\Omega)$ and $\nabla w(y)=A^{\mathrm{T}}\nabla v(Ay)$, and the substitution $x=Ay$ gives $\int_\Omega h_Q(\nabla v)^p\,dx=(\det A)\int|\nabla w|^p\,dy$ and $\int_\Omega|v|^q\,dx=(\det A)\int|w|^q\,dy$, which proves the first two identities. In the same way, $u\in\mathcal{W}(K)$ if and only if $u\circ A\in\mathcal{W}(A^{-1}K)$, which proves the third one. If $A=tA_0$, then $\det A=t^n$ and $\operatorname{cap}_p(A^{-1}K)=t^{p-n}\operatorname{cap}_p(A_0^{-1}K)$ by the scaling of the capacity.

(iv) Since $h_Q$ is even, the infimum defining $\Lambda^{Q}_{p,q}(\Omega)$ may be taken over nonnegative functions. Let $s>0$ satisfy $|sQ|=|\Omega|$, and $v\in W^{1,p}_0(\Omega)$ be a non-negative function. The convex symmetrization of $v$ belongs to $W^{1,p}_0(sQ)$, has the same $L^q$ norm and has smaller or equal anisotropic energy by \cite[Theorem 3.1]{AFLT1997}. Hence $\Lambda^{Q}_{p,q}(\Omega)\geq\Lambda^{Q}_{p,q}(sQ)=s^{\,n-p-np/q}\,\Lambda^{Q}_{p,q}(Q)$. If $\rho$ is the radius of $\Omega^\star$, then $s^n|Q|=\rho^n\omega_n$, and by (i) and the scaling $\Lambda_{p,q}(\rho B^n)=\rho^{\,n-p-np/q}\,\Lambda_{p,q}(B^n)$ the right-hand side equals $(|Q|/\omega_n)^{p/n}\Lambda_{p,q}(\Omega^\star)$. The case $q=1$ gives the inequality for $T^{Q}_p$.
\end{proof}

\subsection{The case \texorpdfstring{$p=2$}{p=2}}\label{sec:p2}
For $p=2$ the anisotropic norms $H_f$ are Euclidean norms after a linear change of variables, and the affine quantities become extremal values of Euclidean ones. In this case the constant in \cite[Theorem 1.2]{HuangLi2016} equals $1$, so that the identity $\mathcal{E}_2(f)^2=\min_{A\in\mathrm{SL}(n)}\|\nabla(f\circ A)\|_2^2$ in the following proposition is contained in the work of Huang and Li \cite[Theorems 1.1 and 1.2]{HuangLi2016}. The form $n\det(Q_f)^{1/n}$, and its use to write variational problems for the affine Laplacian as minimizations over pairs $(v,A)$, appear in the work of Schindler and Tintarev \cite[Lemma 2.1 and Corollary 2.2]{SchindlerTintarev2018}. 

\begin{proposition}\label{prop:p2}
Let $f\in W^{1,2}_0(\Omega)\setminus\{0\}$ for a bounded open set $\Omega$, or $f\in D^{1,2}(\mathbb{R}^n)\setminus\{0\}$ with $n\geq3$, and let $Q_f=\int_{\mathbb{R}^n}\nabla f\otimes\nabla f\,dx$. Then
\[
\mathcal{E}_2(f)^2=n\det(Q_f)^{1/n}=\min_{A\in\mathrm{SL}(n)}\|\nabla(f\circ A)\|_2^2,
\]
and the minimum is attained exactly when $AA^{\mathrm{T}}=\det(Q_f)^{1/n}Q_f^{-1}$. Consequently, for every bounded open set $\Omega$ and every $1\leq q<2^*$,
\[
\Lambda^{\mathcal{A}}_{2,q}(\Omega)=\min_{A\in\mathrm{SL}(n)}\Lambda_{2,q}(A\Omega),\qquad T^{\mathcal{A}}_2(\Omega)=\max_{A\in\mathrm{SL}(n)}T_2(A\Omega),
\]
and, if $n\geq3$, $\operatorname{cap}^{\mathcal{A}}_2(K)=\inf_{A\in\mathrm{SL}(n)}\operatorname{cap}_2(AK)$ for every compact set $K$, with the infimum attained if $|K|>0$.
\end{proposition}
The formulas for $\Lambda^{\mathcal{A}}_{2,q}$, $T^{\mathcal{A}}_2$ and $\operatorname{cap}^{\mathcal{A}}_2$ follow by applying the identity to extremals, which exist by Proposition \ref{prop:attain} and Theorem \ref{thm:potential}.

\section{Existence of extremals}\label{sec:existence}

To prove the existence of extremals we first need lower bounds for the directional energies $\|\nabla_\xi f\|_p$ along minimizing sequences, uniform in $\xi\in S^{n-1}$. On a bounded domain it follows from the width of $\Omega$ (Proposition \ref{prop:energy}(ii)).  For the affine capacity of a compact set $K$ they follow from an estimate along the lines that meet $K$ (Lemma \ref{lem:slicing}).

For $p\geq n$ the affine Sobolev inequality is not available at the exponent $p$; the following lemma allows us to use it at a smaller exponent.

\begin{lemma}\label{lem:lower}
Let $\Omega\subset\mathbb{R}^n$ be a bounded open set and $1<r<p$. For every $f\in W^{1,p}_0(\Omega)$,
\[
\mathcal{E}_r(f)\leq\frac{c_{n,r}}{c_{n,p}}\,|\Omega|^{\frac1r-\frac1p}\,\mathcal{E}_p(f).
\]
\end{lemma}

\begin{proof}
Since $\nabla_\xi f$ vanishes outside $\Omega$, H\"older's inequality gives $\|\nabla_\xi f\|_r\leq|\Omega|^{\frac1r-\frac1p}\|\nabla_\xi f\|_p$ for every $\xi\in S^{n-1}$. Raising to the power $-n$ and integrating over $S^{n-1}$ gives the claim.
\end{proof}

\begin{proposition}\label{prop:attain}
Let $p>1$, $1\leq q<p^*$, and let $\Omega\subset\mathbb{R}^n$ be a bounded open set. Then $\Lambda^{\mathcal{A}}_{p,q}(\Omega)>0$, and it is attained by a nonnegative function.
\end{proposition}

For $q=p$ this is \cite[Theorem 3]{HJM2021}. For $1\leq q<p$ it also follows from \cite[Theorem 4]{LM2024CV} and, for $p<n$, from \cite[Theorem 4.2]{LR2025}; for $p<q<p^*$ and $p<n$ it is contained in \cite[Corollary 1.1]{LM2024ACV}; for $p=2$ and $2<q<2^*$ see also \cite[Theorem 3.7]{SchindlerTintarev2018}. The following argument covers all cases, including $p\geq n$ and $q>p$.

\begin{proof}
Let $(v_k)\subset W^{1,p}_0(\Omega)$ be a minimizing sequence with $\|v_k\|_q=1$; replacing $v_k$ by $|v_k|$ we may assume $v_k\geq0$. We claim that $\inf_k\|v_k\|_p>0$. If $q\leq p$, this follows from H\"older's inequality, $1=\|v_k\|_q\leq|\Omega|^{\frac1q-\frac1p}\|v_k\|_p$. If $q>p$, choose $s\in(q,\infty)$ as follows: $s=p^*$ if $p<n$, and $s=r^*=nr/(n-r)$ for some $r\in(1,n)$ with $r^*>q$ if $p\geq n$. By the affine Sobolev inequality \eqref{eq:affineSobolev} at the exponent $p$, respectively at the exponent $r$ combined with Lemma \ref{lem:lower}, the sequence $(v_k)$ is bounded in $L^s$. Writing $\frac1q=\frac\lambda p+\frac{1-\lambda}{s}$ with $\lambda\in(0,1)$, the interpolation inequality $1=\|v_k\|_q\leq\|v_k\|_p^\lambda\|v_k\|_s^{1-\lambda}$ proves the claim. Proposition \ref{prop:energy}(ii) now shows that $(v_k)$ is bounded in $W^{1,p}_0(\Omega)$. By the Rellich--Kondrachov theorem, a subsequence converges weakly in $W^{1,p}_0(\Omega)$ and strongly in $L^q(\Omega)$ to a function $u\geq0$ with $\|u\|_q=1$, and by Proposition \ref{prop:energy}(iii), $\mathcal{E}_p(u)^p\leq\liminf_k\mathcal{E}_p(v_k)^p=\Lambda^{\mathcal{A}}_{p,q}(\Omega)$. Hence $u$ is a minimizer, and $\Lambda^{\mathcal{A}}_{p,q}(\Omega)=\mathcal{E}_p(u)^p>0$ by Proposition \ref{prop:energy}(ii).
\end{proof}

For the affine capacity, the width bound is replaced by the following estimate.

\begin{lemma}\label{lem:slicing}
Let $K\subset\mathbb{R}^n$ be a compact set with $|K|>0$, and let $f\in D^{1,p}(\mathbb{R}^n)$ satisfy $f=1$ a.e. on $K$ and $\|f\|_{p^*}\leq m$. Then, for every $\xi\in S^{n-1}$,
\[
\|\nabla_\xi f\|_p^p\geq c_K=2^{-p}\big(2^{p^*}m^{p^*}\big)^{1-p}\Big(\frac{|K|}{\operatorname{diam}K}\Big)^p.
\]
\end{lemma}

\begin{proof}
Fix $\xi\in S^{n-1}$, write $x=x'+\tau\xi$ with $x'\in\xi^\perp$, and let $g_{x'}(\tau)=f(x'+\tau\xi)$. For a.e. $x'$, the function $g_{x'}$ is locally absolutely continuous with $g_{x'}'(\tau)=\nabla_\xi f(x'+\tau\xi)$ for a.e. $\tau$, and $g_{x'}=1$ a.e. on $\{\tau:x'+\tau\xi\in K\}$; by Fubini's theorem, $\|\nabla_\xi f\|_p^p=\int_{\xi^\perp}\int_{\mathbb{R}}|g_{x'}'|^p\,d\tau\,dx'$. Let $\ell_K(x')$ and $\ell(x')$ be the lengths of $\{\tau:x'+\tau\xi\in K\}$ and of $\{\tau:g_{x'}(\tau)\geq\frac12\}$, and let $G=\{x'\in\xi^\perp:\ell_K(x')>0\}$. Since $\ell_K\leq\operatorname{diam}K$, we have $|K|=\int_G\ell_K\,dx'\leq|G|\operatorname{diam}K$, where $|G|$ is the $(n-1)$-dimensional measure of $G$; moreover $\int_{\xi^\perp}\ell\,dx'=|\{f\geq\frac12\}|\leq2^{p^*}m^{p^*}$.

Fix $x'\in G$ for which $g=g_{x'}$ has the properties above and $\ell(x')<\infty$. Since $g=1$ a.e. on $\{\tau:x'+\tau\xi\in K\}$, we have $\ell(x')\geq\ell_K(x')>0$, and, $g$ being continuous, $g(\tau_0)=1$ for some $\tau_0$. As $\{g\geq\frac12\}$ has length $\ell(x')$, the closed interval $[\tau_0,\tau_0+\ell(x')]$ contains a point $\tau_1$ with $g(\tau_1)\leq\frac12$. By H\"older's inequality,
\[
\tfrac12\leq g(\tau_0)-g(\tau_1)\leq\ell(x')^{1-\frac1p}\Big(\int_{\mathbb{R}}|g'|^p\,d\tau\Big)^{\frac1p},
\]
and therefore $\int_{\mathbb{R}}|g'|^p\,d\tau\geq2^{-p}\ell(x')^{1-p}$. Integrating over $G$ and applying Jensen's inequality to the convex function $t\mapsto t^{1-p}$, we obtain
\[
\|\nabla_\xi f\|_p^p\geq2^{-p}\int_G\ell^{1-p}\,dx'\geq2^{-p}|G|^p\Big(\int_G\ell\,dx'\Big)^{1-p}\geq2^{-p}\Big(\frac{|K|}{\operatorname{diam}K}\Big)^p\big(2^{p^*}m^{p^*}\big)^{1-p}.
\qedhere
\]
\end{proof}

\begin{proof}[Proof of Theorem \ref{thm:potential}]
Let $(f_j)\subset\mathcal{W}(K)$ be a minimizing sequence for $\operatorname{cap}^{\mathcal{A}}_p(K)$. Truncation does not increase the directional energies, so we may assume $0\leq f_j\leq1$, hence $f_j=1$ a.e. on $K$; and, for $j$ large, Proposition \ref{prop:energy}(iv) gives $\|f_j\|_{p^*}\leq m=S_{n,p}^{-1}(\operatorname{cap}^{\mathcal{A}}_p(K)+1)^{1/p}$. By Lemma \ref{lem:slicing}, $\|\nabla_\xi f_j\|_p\geq c_K^{1/p}$ for all $\xi$, that is, $\Pi^*_pf_j\subseteq RB^n$ with $R=c_K^{-1/p}$. On the other hand, since $\mathcal{E}_p(f_j)^p\leq\operatorname{cap}^{\mathcal{A}}_p(K)+1$ for $j$ large,
\[
|\Pi^*_pf_j|=\frac1n\Big(\frac{c_{n,p}}{\mathcal{E}_p(f_j)}\Big)^n\geq v_0=\frac{c_{n,p}^{\,n}}{n}\big(\operatorname{cap}^{\mathcal{A}}_p(K)+1\big)^{-n/p}>0.
\]
Hence, $\Pi^*_pf_j$ is an origin-symmetric convex body contained in $RB^n$ with volume at least $v_0$ and, therefore, it contains $\rho_0B^n$, where $\rho_0=v_0/(2\omega_{n-1}R^{n-1})$; indeed, if $r$ is the distance from the origin to its boundary, then, by symmetry, it lies in a slab of width $2r$, whose intersection with $RB^n$ has volume at most $2r\,\omega_{n-1}R^{n-1}$. As a consequence, 
\[
c_K^{1/p}\leq\|\nabla_\xi f_j\|_p\leq\rho_0^{-1}\qquad\text{for all }\xi\in S^{n-1}\text{ and all large }j,
\]
and, by \eqref{eq:average}, $(f_j)$ is bounded in $D^{1,p}(\mathbb{R}^n)$. A subsequence converges weakly to a function $u$, which belongs to $\mathcal{W}(K)$ because this class is weakly closed. For every $\xi$, $\|\nabla_\xi u\|_p\leq\liminf_j\|\nabla_\xi f_j\|_p$, that is, $\rho_{\Pi^*_pu}\geq\limsup_j\rho_{\Pi^*_pf_j}$ pointwise on $S^{n-1}$. These radial functions are bounded by $R$, so the reverse Fatou lemma gives $|\Pi^*_pu|\geq\limsup_j|\Pi^*_pf_j|$, that is, $\mathcal{E}_p(u)^p\leq\liminf_j\mathcal{E}_p(f_j)^p=\operatorname{cap}^{\mathcal{A}}_p(K)$. Thus $u$ is a minimizer.

Now let $u$ be any minimizer. Its truncation $v=\min\{\max\{u,0\},1\}$ belongs to $\mathcal{W}(K)$ and satisfies $\|\nabla_\xi v\|_p\leq\|\nabla_\xi u\|_p$ for every $\xi$. Since $\mathcal{E}_p(v)\geq\mathcal{E}_p(u)$ and both functions of $\xi$ are continuous, equality holds for every $\xi$, so that $\nabla u=0$ a.e. on $\{u<0\}\cup\{u>1\}$ and hence $u=v$. Thus $0\leq u\leq1$ and $\tilde u=1$ q.e. on $K$; in particular $\nabla u=0$ a.e. on $K$.

Let $H=H_u$. For every $w\in\mathcal{W}(K)$, Lemma \ref{lem:frozen} gives
\[
\int_{\mathbb{R}^n}H(\nabla w)^p\,dx\geq\mathcal{E}_p(w)^p\geq\operatorname{cap}^{\mathcal{A}}_p(K)=\mathcal{E}_p(u)^p=\int_{\mathbb{R}^n}H(\nabla u)^p\,dx,
\]
so $u$ minimizes the anisotropic energy of $H$ on $\mathcal{W}(K)$. Since $H$ is a strictly convex norm of class $C^1$ on $\mathbb{R}^n\setminus\{0\}$ (Lemma \ref{lem:admissible}) and $p>1$, the function $H^p$ is strictly convex, and $u$ is the only minimizer. The functional $w\mapsto\int H(\nabla w)^p\,dx$ is of class $C^1$ on $D^{1,p}(\mathbb{R}^n)$, with derivative at $u$ equal to $p\,\mu$, where
\[
\langle\mu,\varphi\rangle=\int_{\mathbb{R}^n}H(\nabla u)^{p-1}\langle\nabla H(\nabla u),\nabla\varphi\rangle\,dx,
\]
and $\mu=\Delta^{\mathcal{A}}_pu$ in the sense of distributions, by Lemma \ref{lem:frozen}. If $\varphi\in C^\infty_c(\mathbb{R}^n)$ and $\varphi\geq0$, then $u+t\varphi\in\mathcal{W}(K)$ for $t>0$, so $\langle\mu,\varphi\rangle\geq0$; thus $\mu$ is a nonnegative distribution, that is, a nonnegative Radon measure. If $\varphi\in C^\infty_c(\mathbb{R}^n\setminus K)$, then $u+t\varphi\in\mathcal{W}(K)$ for every $t\in\mathbb{R}$, so $\langle\mu,\varphi\rangle=0$; and if $\varphi\in C^\infty_c(\operatorname{int}K)$, then $\langle\mu,\varphi\rangle=0$ because $\nabla u=0$ a.e. on $K$. Hence $\Delta^{\mathcal{A}}_pu=0$ in $\mathbb{R}^n\setminus K$ and $\mu$ is supported on $\partial K$.

Finally, let $\psi\in C^\infty_c(\mathbb{R}^n)$ be equal to $1$ on a neighborhood of $K$, so that $\mu(\mathbb{R}^n)=\langle\mu,\psi\rangle$. The function $w=\psi-u$ satisfies $\tilde w=0$ q.e. on $K$, so it is a limit in $D^{1,p}(\mathbb{R}^n)$ of functions in $C^\infty_c(\mathbb{R}^n\setminus K)$ (Section \ref{sec:functions}), on which $\mu$ vanishes; as $\mu$ is continuous on $D^{1,p}(\mathbb{R}^n)$, $\langle\mu,w\rangle=0$. Therefore, by Euler's identity $\langle\nabla H(\zeta),\zeta\rangle=H(\zeta)$,
\[
\mu(\mathbb{R}^n)=\langle\mu,u\rangle=\int_{\mathbb{R}^n}H(\nabla u)^p\,dx=\operatorname{cap}^{\mathcal{A}}_p(K).
\qedhere
\]
\end{proof}

\begin{remark}\label{rem:filling}
Let $\hat K$ be the union of $K$ with the bounded components of $\mathbb{R}^n\setminus K$. Then $\operatorname{cap}^{\mathcal{A}}_p(K)=\operatorname{cap}^{\mathcal{A}}_p(\hat K)$. Indeed, if $\varphi\in C^\infty_c(\mathbb{R}^n)$ and $\varphi\geq1$ on $K$, the function equal to $\min\{\max\{\varphi,0\},1\}$ on the unbounded component of $\mathbb{R}^n\setminus K$ and to $1$ on $\hat K$ is Lipschitz, has compact support and directional energies not larger than those of $\varphi$, and a mollification of a slightly larger multiple of it is admissible for $\hat K$. Applied to a sequence approximating a minimizer for $K$, the same construction shows that every such minimizer equals $1$ a.e. on $\hat K$. Consequently, Theorem \ref{thm:potential} holds whenever $|\hat K|>0$, for instance when $K$ is the boundary of a bounded open set. On the other hand, the affine capacity of a compact set with $|\hat K|=0$ may vanish, as for subsets of hyperplanes \cite[Theorem 2.3(i)]{XiaoZhang2016}, and we do not know whether it is attained when it is positive.
\end{remark}

\section{Isoperimetric inequalities}\label{sec:iso}

Equality in the P\'olya--Szeg\H{o} inequality $\|\nabla f^\star\|_p\leq\|\nabla f\|_p$, for nonnegative $f\in W^{1,p}(\mathbb{R}^n)$, forces $f$ to be a translate of $f^\star$ provided that $\big|\{x:\nabla f^\star(x)=0,\ 0<f^\star(x)<\operatorname{ess\,sup}f\}\big|=0$, by a theorem of Brothers and Ziemer \cite{BrothersZiemer1988}.

The following lemma provides this non-degeneracy condition for extremals on a ball. For $q=p$ it is classical \cite{FMP2009}, and for $1<q<p^*$ it is observed in \cite[Section 3]{FranzinaLamberti2010}; we include the proof, which also covers $q=1$.

\begin{lemma}\label{lem:radial}
Let $p>1$, $1\leq q<p^*$ and $R>0$, and let $w\in W^{1,p}_0(B_R)$ be a nonnegative, radially symmetric and radially nonincreasing function with $\|w\|_q=1$ and $\|\nabla w\|_p^p=\Lambda_{p,q}(B_R)$. Write $w(x)=\varphi(|x|)$, where $\varphi$ is nonincreasing and locally absolutely continuous on $(0,R)$. Then $\varphi>0$ on $(0,R)$ and $\varphi'<0$ a.e. in $(0,R)$. In particular, $\{w>0\}=B_R$ up to a null set, and
\[
\big|\{x\in B_R:\ \nabla w(x)=0,\ 0<w(x)<\operatorname{ess\,sup}w\}\big|=0.
\]
\end{lemma}

\begin{proof}
Let $\Lambda=\Lambda_{p,q}(B_R)$. We first show that, with the convention $w^0=1$,
\begin{equation}\label{eq:EL}
\int_{B_R}|\nabla w|^{p-2}\langle\nabla w,\nabla\psi\rangle\,dx=\Lambda\int_{B_R}w^{q-1}\psi\,dx\qquad\text{for every }\psi\in W^{1,p}_0(B_R).
\end{equation}
If $q>1$, the function $t\mapsto\|\nabla(w+t\psi)\|_p^p-\Lambda\|w+t\psi\|_q^p$ is nonnegative, vanishes at $t=0$ and is differentiable there, so its derivative at $t=0$ vanishes, which is \eqref{eq:EL}. If $q=1$, then $\int w=1$, and for $|t|$ small the function $g(t)=\|\nabla(w+t\psi)\|_p^p-\Lambda\big(\int(w+t\psi)\,dx\big)^p$ satisfies $g(t)\geq\|\nabla(w+t\psi)\|_p^p-\Lambda\|w+t\psi\|_1^p\geq0=g(0)$; hence $g'(0)=0$, which is again \eqref{eq:EL}.

Now take $\psi(x)=\eta(|x|)$ with $\eta\in C^\infty([0,R])$ constant near $0$ and vanishing near $R$. In polar coordinates, \eqref{eq:EL} becomes
\begin{equation}\label{eq:radialEL}
\int_0^RG\,\eta'\,dr=\int_0^Rh\,\eta\,dr,\qquad G(r)=r^{n-1}|\varphi'(r)|^{p-2}\varphi'(r),\quad h(r)=\Lambda\,r^{n-1}\varphi(r)^{q-1}.
\end{equation}
By H\"older's inequality, $G$ and $h$ belong to $L^1(0,R)$. Choosing $\eta\in C^\infty_c(0,R)$ in \eqref{eq:radialEL} shows that $G'=-h$ in the sense of distributions on $(0,R)$; hence $G$ coincides a.e. with an absolutely continuous function, still denoted by $G$, which extends continuously to $[0,R]$. Integrating by parts in \eqref{eq:radialEL} with $\eta(0)=1$ yields $G(0)=0$, and therefore
\[
G(r)=-\Lambda\int_0^rs^{n-1}\varphi(s)^{q-1}\,ds\qquad\text{for every }r\in(0,R).
\]
Since $w\neq0$ and $\varphi$ is nonincreasing, $\varphi>0$ near $0$, so the integral is positive for every $r\in(0,R)$; thus $G<0$ on $(0,R)$, that is, $\varphi'<0$ a.e. in $(0,R)$. If $\varphi(r_1)=0$ for some $r_1<R$, then $\varphi$ vanishes on $[r_1,R)$ and $\varphi'=0$ there, a contradiction. Hence $\varphi>0$ on $(0,R)$. The last assertion follows because $|\nabla w(x)|=|\varphi'(|x|)|$ for a.e. $x$.
\end{proof}

\begin{proof}[Proof of Theorem \ref{thm:iso}]
Let $v\in W^{1,p}_0(\Omega)\setminus\{0\}$. By Lemma \ref{lem:frozen}, Lemma \ref{lem:Wulff}(iv) with $Q=\Gamma_pL_v$, and the $L_p$ Busemann--Petty centroid inequality \eqref{eq:BP}, since $|L_v|=\omega_n$,
\begin{equation}\label{eq:isochain}
\mathcal{E}_p(v)^p=\int_\Omega H_v(\nabla v)^p\,dx\geq\Lambda^{\Gamma_pL_v}_{p,q}(\Omega)\,\|v\|_q^p\geq\Big(\frac{|\Gamma_pL_v|}{|L_v|}\Big)^{\frac pn}\Lambda_{p,q}(\Omega^\star)\,\|v\|_q^p\geq\Lambda_{p,q}(\Omega^\star)\,\|v\|_q^p.
\end{equation}
Taking the infimum over $v$ proves \eqref{eq:iso}.

Suppose that $\Omega\subseteq E$ and $C_p(E\setminus\Omega)=0$ for an open ellipsoid $E$. By Section \ref{sec:functions}, $\Lambda^{\mathcal{A}}_{p,q}(\Omega)=\Lambda^{\mathcal{A}}_{p,q}(E)$, and $|E|=|\Omega|$ because sets of zero capacity are null. Hence $E=A\,\Omega^\star+x_0$ for some $A\in\mathrm{SL}(n)$ and $x_0\in\mathbb{R}^n$, and \eqref{eq:invariance} gives $\Lambda^{\mathcal{A}}_{p,q}(\Omega)=\Lambda^{\mathcal{A}}_{p,q}(\Omega^\star)\leq\Lambda_{p,q}(\Omega^\star)$. By \eqref{eq:iso}, equality holds.

Conversely, suppose that equality holds in \eqref{eq:iso}, and let $u\geq0$ be a minimizer for $\Lambda^{\mathcal{A}}_{p,q}(\Omega)$ with $\|u\|_q=1$, given by Proposition \ref{prop:attain}. Then \eqref{eq:isochain} with $v=u$ is a chain of equalities, so that $|\Gamma_pL_u|=|L_u|$. By the equality case of \eqref{eq:BP}, $L_u$ is an origin-symmetric ellipsoid, and since $\Gamma_p$ commutes with linear maps and $\Gamma_pB^n=B^n$, we have $\Gamma_pL_u=L_u=A_0B^n$ for some $A_0\in\mathrm{SL}(n)$. Let $D=A_0^{-1}\Omega$ and $w=u\circ A_0\in W^{1,p}_0(D)$. By Lemma \ref{lem:Wulff}(iii), $\|w\|_q=1$ and $\|\nabla w\|_p^p=\int_\Omega h(\Gamma_pL_u,\nabla u)^p\,dx=\Lambda_{p,q}(\Omega^\star)$, and $D^\star=\Omega^\star$. By the P\'olya--Szeg\H{o} principle, $w^\star\in W^{1,p}_0(\Omega^\star)$, $\|w^\star\|_q=1$ and $\|\nabla w^\star\|_p\leq\|\nabla w\|_p$; hence $w^\star$ is a nonnegative, radially nonincreasing minimizer for $\Lambda_{p,q}(\Omega^\star)$, and $\|\nabla w^\star\|_p=\|\nabla w\|_p$. Let $R$ be the radius of $\Omega^\star$ and write $w^\star(x)=\varphi(|x|)$. By Lemma \ref{lem:radial}, $w^\star$ satisfies the non-degeneracy condition of the theorem of Brothers and Ziemer, since $\operatorname{ess\,sup}w=\operatorname{ess\,sup}w^\star$, and that theorem gives $w=w^\star(\cdot+y_0)$ a.e. for some $y_0\in\mathbb{R}^n$. With $A=A_0$ and $x_0=A_0y_0$, this reads
\[
u(x)=\varphi\big(|A^{-1}(x+x_0)|\big)\qquad\text{for a.e. }x\in\mathbb{R}^n.
\]
Let $E'=A\,\Omega^\star-x_0$, an open ellipsoid with $|E'|=|\Omega|$ and center $z=-x_0$. By Lemma \ref{lem:radial}, the right-hand side is positive and continuous in $E'\setminus\{z\}$. Being continuous, it is quasicontinuous there, so it coincides q.e. in $E'\setminus\{z\}$ with the quasicontinuous representative $\tilde u$ of $u$. On the other hand, $\tilde u=0$ q.e. in $\mathbb{R}^n\setminus\Omega$ by \eqref{eq:W0cap}. Therefore $C_p\big((E'\setminus\Omega)\setminus\{z\}\big)=0$. If $p\leq n$, points have zero Sobolev $p$-capacity. If $p>n$, $\tilde u$ is continuous and vanishes at every point of $\mathbb{R}^n\setminus\Omega$, while $\tilde u(z)=\lim_{r\to0^+}\varphi(r)>0$, so $z\in\Omega$. In both cases $C_p(E'\setminus\Omega)=0$; in particular $|E'\setminus\Omega|=0$, and since $|\Omega|=|E'|$, also $|\Omega\setminus E'|=0$. The open set $\Omega\setminus\overline{E'}$ is then null, hence empty, so $\Omega\subseteq\overline{E'}$ and, $\Omega$ being open, $\Omega\subseteq E'$. Thus $\Omega$ is an ellipsoid up to a set of zero Sobolev $p$-capacity.
\end{proof}

\begin{proof}[Proof of Corollary \ref{cor:SV}]
The first inequality is \eqref{eq:invariance}. Since $T^{\mathcal{A}}_p(\Omega)=\Lambda^{\mathcal{A}}_{p,1}(\Omega)^{-1/(p-1)}$ and $T_p(\Omega^\star)=\Lambda_{p,1}(\Omega^\star)^{-1/(p-1)}$, the second inequality and its equality cases are the case $q=1$ of Theorem \ref{thm:iso}. The value of $T_p(\Omega^\star)$ follows from \eqref{eq:TBall} with $R=(|\Omega|/\omega_n)^{1/n}$.
\end{proof}

\begin{proof}[Proof of Theorem \ref{thm:isocap}]
The first inequality follows from $\mathcal{E}_p\leq\|\nabla\cdot\|_p$. For the second we may assume $|K|>0$, since otherwise $\operatorname{cap}_p(K^\star)=0$; then every $u\in\mathcal{W}(K)$ is nonzero. By Lemma \ref{lem:frozen}, Lemma \ref{lem:Wulff}(ii) with $Q=\Gamma_pL_u$, and the $L_p$ Busemann--Petty centroid inequality \eqref{eq:BP}, since $|L_u|=\omega_n$,
\begin{equation}\label{eq:capchain}
\mathcal{E}_p(u)^p=\int_{\mathbb{R}^n}H_u(\nabla u)^p\,dx\geq\operatorname{cap}^{\Gamma_pL_u}_p(K)\geq\Big(\frac{|\Gamma_pL_u|}{|L_u|}\Big)^{\frac pn}\operatorname{cap}_p(K^\star)\geq\operatorname{cap}_p(K^\star),
\end{equation}
and the second inequality follows by taking the infimum over $u$.

If $E\subseteq K$ is a closed ellipsoid and $C_p(K\setminus E)=0$, then $\mathcal{W}(K)=\mathcal{W}(E)$ and $|K|=|E|$. Hence $\operatorname{cap}^{\mathcal{A}}_p(K)=\operatorname{cap}^{\mathcal{A}}_p(E)=\operatorname{cap}^{\mathcal{A}}_p(K^\star)$ by the affine invariance, and $\operatorname{cap}^{\mathcal{A}}_p(K^\star)=\operatorname{cap}_p(K^\star)$ by the two inequalities applied to $K^\star$.

Conversely, suppose that $\operatorname{cap}^{\mathcal{A}}_p(K)=\operatorname{cap}_p(K^\star)$ and let $u$ be a minimizer given by Theorem \ref{thm:potential}, so that $0\leq u\leq1$. Since $\mathcal{E}_p(u)^p=\operatorname{cap}_p(K^\star)$, \eqref{eq:capchain} is a chain of equalities for this $u$. Hence $|\Gamma_pL_u|=|L_u|$ and, as in the proof of Theorem \ref{thm:iso}, $\Gamma_pL_u=A_0B^n$ for some $A_0\in\mathrm{SL}(n)$. By Lemma \ref{lem:Wulff}(iii), $w=u\circ A_0$ belongs to $\mathcal{W}(A_0^{-1}K)$ and $\|\nabla w\|_p^p=\operatorname{cap}_p(K^\star)$. Since $0\leq w\leq1$ and $w=1$ a.e. on $A_0^{-1}K$, the rearrangement $w^\star$ belongs to $D^{1,p}(\mathbb{R}^n)$ and satisfies $w^\star\geq1$ on $K^\star$, because $\{w^\star\geq1\}$ is a closed ball of volume $|\{w\geq1\}|\geq|A_0^{-1}K|=|K|$; being radially nonincreasing and continuous away from the origin, which has zero Sobolev $p$-capacity, $w^\star$ belongs to $\mathcal{W}(K^\star)$. Moreover $\|\nabla w^\star\|_p\leq\|\nabla w\|_p$ by the P\'olya--Szeg\H{o} principle; hence $w^\star$ minimizes $\|\nabla\cdot\|_p^p$ on $\mathcal{W}(K^\star)$, and by the uniqueness of the capacitary potential, $w^\star=v_R$, where $K^\star=\overline{B}_R$ and $v_R$ is given in Section \ref{sec:functions}. For $0<t<1$, the function $(w-t)_+$ belongs to $W^{1,p}(\mathbb{R}^n)$, because $|\{w>t\}|<\infty$, and its symmetric decreasing rearrangement is $(v_R-t)_+$. Since $\|\nabla w\|_p^p=\|\nabla(w-t)_+\|_p^p+\|\nabla\min\{w,t\}\|_p^p$, and likewise for $v_R$, the P\'olya--Szeg\H{o} inequality for both terms and $\|\nabla w\|_p=\|\nabla v_R\|_p$ give $\|\nabla(w-t)_+\|_p=\|\nabla(v_R-t)_+\|_p$. As $v_R$ is strictly decreasing in $|x|$ where $0<v_R<1$, the theorem of Brothers and Ziemer gives $(w-t)_+=(v_R-t)_+(\cdot+y_t)$ a.e., and $y_t$ does not depend on $t$, because for $s>t$ the superlevel set $\{w>s\}$ is a ball centered at $-y_t$. Letting $t\to0$, we get $w=v_R(\cdot+y_0)$ a.e., that is, $u(x)=v_R\big(A^{-1}(x+x_0)\big)$ for a.e. $x$, with $A=A_0$ and $x_0=A_0y_0$. The right-hand side is continuous, equal to $1$ on the closed ellipsoid $E=A\overline{B}_R-x_0$ and smaller than $1$ outside $E$; being continuous, it is the quasicontinuous representative of $u$. Since $\tilde u=1$ q.e. on $K$, we obtain $C_p(K\setminus E)=0$, hence $|K\setminus E|=0$ and, as $|E|=|K|$, also $|E\setminus K|=0$. Thus $K\cap E$ is a closed subset of $E$ of full measure, which forces $E\subseteq K$.
\end{proof}

Huang and Li \cite[Theorems 1.1 and 1.2]{HuangLi2016} proved that, for $f\in W^{1,p}(\mathbb{R}^n)$, the minimum of $\|\nabla(f\circ A)\|_p$ over $A\in\mathrm{SL}(n)$ is attained and
\begin{equation}\label{eq:HuangLi}
m_{n,p}\min_{A\in\mathrm{SL}(n)}\|\nabla(f\circ A)\|_p\leq\mathcal{E}_p(f)\leq\min_{A\in\mathrm{SL}(n)}\|\nabla(f\circ A)\|_p,
\end{equation}
where, in the normalization \eqref{eq:Ep},
\[
m_{n,p}=\frac{\pi^{\frac1{2p}+\frac12}\,\Gamma(\frac{n+p}2)^{\frac1p}\,\Gamma(1+\frac np)^{\frac1n}}{2^{\frac1p+1}\,\Gamma(1+\frac n2)^{\frac1n+\frac1p}\,\Gamma(\frac{p+1}2)^{\frac1p}\,\Gamma(1+\frac1p)}.
\]
The constant is sharp \cite[Theorem 4.4]{HuangLi2016}; $m_{n,2}=1$ and $m_{n,p}<1$ for $p\neq2$. Taking infima over admissible functions gives the following comparison. Its last assertion is due to Xiao and Zhang \cite[Theorem 2.3(i)]{XiaoZhang2016}.

\begin{proposition}\label{prop:compare}
Let $1<p<n$ and let $K\subset\mathbb{R}^n$ be compact. Then
\[
m_{n,p}^p\inf_{A\in\mathrm{SL}(n)}\operatorname{cap}_p(AK)\leq\operatorname{cap}^{\mathcal{A}}_p(K)\leq\inf_{A\in\mathrm{SL}(n)}\operatorname{cap}_p(AK).
\]
If $p=2$, then $\operatorname{cap}^{\mathcal{A}}_2(K)=\inf_{A\in\mathrm{SL}(n)}\operatorname{cap}_2(AK)$, and the infimum is attained when $|K|>0$. If $K$ is contained in a hyperplane, then $\operatorname{cap}^{\mathcal{A}}_p(K)=0$.
\end{proposition}

\begin{proof}
For $A\in\mathrm{SL}(n)$, a function $u$ belongs to $\mathcal{W}(K)$ if and only if $u\circ A^{-1}\in\mathcal{W}(AK)$, and $\mathcal{E}_p(u\circ A^{-1})=\mathcal{E}_p(u)$; hence $\operatorname{cap}^{\mathcal{A}}_p(K)=\operatorname{cap}^{\mathcal{A}}_p(AK)\leq\operatorname{cap}_p(AK)$, which is the upper bound.

For the lower bound, let $\varphi\in C^\infty_c(\mathbb{R}^n)$ with $\varphi\geq1$ on $K$, and let $A\in\mathrm{SL}(n)$ attain the minimum in \eqref{eq:HuangLi} for $f=\varphi$. Then $\varphi\circ A\geq1$ on $A^{-1}K$, so that
\[
\mathcal{E}_p(\varphi)^p\geq m_{n,p}^p\|\nabla(\varphi\circ A)\|_p^p\geq m_{n,p}^p\operatorname{cap}_p(A^{-1}K)\geq m_{n,p}^p\inf_{A'\in\mathrm{SL}(n)}\operatorname{cap}_p(A'K),
\]
and the lower bound follows by taking the infimum over $\varphi$.

For $p=2$, the remaining assertions are contained in Proposition \ref{prop:p2}.

Finally, if $K$ is contained in a hyperplane, which we may assume to be $e_n^\perp$, let $A_\varepsilon$ be the diagonal matrix with entries $\varepsilon^{1/(n-1)}$ in the first $n-1$ places and $\varepsilon^{-1}$ in the last one. Then $A_\varepsilon\in\mathrm{SL}(n)$, $A_\varepsilon K=\varepsilon^{1/(n-1)}K$, and $\operatorname{cap}_p(A_\varepsilon K)=\varepsilon^{(n-p)/(n-1)}\operatorname{cap}_p(K)\to0$ as $\varepsilon\to0$.
\end{proof}

\section{The affine Kohler-Jobin inequality}\label{sec:KJ}

We begin by recalling the anisotropic extension of \eqref{eq:KJEuclid} proved by Brasco.

\begin{theorem}[{\cite[Theorem 6.1]{Brasco2014}}]\label{thm:Brasco}
Let $p>1$, $1<q<p^*$, and let $Q$ be an origin-symmetric convex body such that $Q^\circ$ is strictly convex and of class $C^1$. Then, for every open set $\Omega\subset\mathbb{R}^n$ of finite measure,
\[
T^{Q}_p(\Omega)^{(p-1)\sigma}\,\Lambda^{Q}_{p,q}(\Omega)\geq T^{Q}_p(Q)^{(p-1)\sigma}\,\Lambda^{Q}_{p,q}(Q),
\]
and equality holds if and only if $\Omega$ is a translate of a dilate of $Q$.
\end{theorem}

In \cite{Brasco2014} the energy is written as $\int\|\nabla v\|_{Q'}^p\,dx$, where $Q'=Q^\circ$ is the unit ball of the norm, the torsional rigidity is the $(p-1)$-th power of ours, and the exponent $\alpha(p,q,n)$ there is our $\sigma$. As all the quantities involved are unchanged by modifications of $\Omega$ of zero Sobolev $p$-capacity, the equality case is to be understood up to such sets.

By Lemma \ref{lem:admissible}, every norm $h(\Gamma_pL,\cdot)$ satisfies the hypotheses of Theorem \ref{thm:Brasco}, since its unit ball is $(\Gamma_pL)^\circ$. The proof of Theorem \ref{thm:KJ} uses that, by \eqref{eq:envelope}, $\Lambda^{\mathcal{A}}_{p,q}$ is the infimum and $T^{\mathcal{A}}_p$ the supremum of the corresponding anisotropic quantities over the same family of norms, and that they enter \eqref{eq:KJ} with exponents of opposite signs.

\begin{proof}[Proof of Theorem \ref{thm:KJ}]
Let $u$ be a minimizer for $\Lambda^{\mathcal{A}}_{p,q}(\Omega)$ with $\|u\|_q=1$ (Proposition \ref{prop:attain}), let $L_u$ be given by \eqref{eq:Lf} and set $Q=\Gamma_pL_u$. By Lemma \ref{lem:frozen},
\[
\Lambda^{\mathcal{A}}_{p,q}(\Omega)=\mathcal{E}_p(u)^p=\int_\Omega h_Q(\nabla u)^p\,dx\geq\Lambda^{Q}_{p,q}(\Omega),
\]
and, since $\mathcal{E}_p(v)^p\leq\int_\Omega h_Q(\nabla v)^p\,dx$ for every $v\in W^{1,p}_0(\Omega)$, also $T^{\mathcal{A}}_p(\Omega)\geq T^{Q}_p(\Omega)$. By Lemma \ref{lem:admissible}, Theorem \ref{thm:Brasco} applies to $Q$, and by Lemma \ref{lem:Wulff}(i) and $|L_u|=\omega_n$ we obtain
\begin{equation}\label{eq:KJchain}
T^{\mathcal{A}}_p(\Omega)^{(p-1)\sigma}\,\Lambda^{\mathcal{A}}_{p,q}(\Omega)\geq T^{Q}_p(\Omega)^{(p-1)\sigma}\,\Lambda^{Q}_{p,q}(\Omega)\geq T^{Q}_p(Q)^{(p-1)\sigma}\,\Lambda^{Q}_{p,q}(Q)=\Big(\frac{|\Gamma_pL_u|}{|L_u|}\Big)^{\theta}\,\kappa,
\end{equation}
where $\theta$ and $\kappa$ are as in \eqref{eq:Wulff}. By the $L_p$ Busemann--Petty centroid inequality \eqref{eq:BP}, the right-hand side is at least $\kappa$, which proves \eqref{eq:KJ}.

If equality holds in \eqref{eq:KJ}, then all the inequalities in \eqref{eq:KJchain} are equalities and $|\Gamma_pL_u|=|L_u|$, because $\theta>0$. By the equality case of \eqref{eq:BP}, $L_u$ is an origin-symmetric ellipsoid, and then $Q=\Gamma_pL_u=L_u$, since $\Gamma_p$ commutes with linear maps and $\Gamma_pB^n=B^n$. Equality in Theorem \ref{thm:Brasco} shows that $\Omega$ is a translate of a dilate of the ellipsoid $Q$, up to a set of zero Sobolev $p$-capacity. Conversely, if $\Omega$ is an ellipsoid up to a set of zero Sobolev $p$-capacity, Theorem \ref{thm:iso} and Corollary \ref{cor:SV} give $\Lambda^{\mathcal{A}}_{p,q}(\Omega)=\Lambda_{p,q}(\Omega^\star)$ and $T^{\mathcal{A}}_p(\Omega)=T_p(\Omega^\star)$, and equality in \eqref{eq:KJ} follows from the invariance of the left-hand side of \eqref{eq:KJEuclid} under dilations.
\end{proof}

\begin{remark}\label{rem:hierarchy}
As in \cite[Remark 1.2]{Brasco2014}, Theorem \ref{thm:KJ} and Corollary \ref{cor:SV} give, for $1<q<p^*$,
\[
\Lambda^{\mathcal{A}}_{p,q}(\Omega)\geq\Big(\frac{T_p(\Omega^\star)}{T^{\mathcal{A}}_p(\Omega)}\Big)^{(p-1)\sigma}\Lambda_{p,q}(\Omega^\star)\geq\Lambda_{p,q}(\Omega^\star),
\]
where the first inequality is \eqref{eq:KJ} together with the dilation invariance of its right-hand side, and the second one is the affine Saint-Venant inequality. Thus the three classical inequalities of Saint-Venant, Faber--Krahn and Kohler-Jobin hold in affine form, with ellipsoids as the only extremals.
\end{remark}

\begin{remark}\label{rem:p2KJ}
Since $T^{\mathcal{A}}_p\geq T_p$ and $\Lambda^{\mathcal{A}}_{p,q}\leq\Lambda_{p,q}$, inequality \eqref{eq:KJ} is not, in general, a strengthening of \eqref{eq:KJEuclid} in the sense in which \eqref{eq:iso} strengthens \eqref{eq:FK}. For $p=2$, Proposition \ref{prop:p2} shows that $\Lambda^{\mathcal{A}}_{2,q}(\Omega)$ and $T^{\mathcal{A}}_2(\Omega)$ are the minimum and the maximum of $\Lambda_{2,q}$ and $T_2$ over the images $A\Omega$, $A\in\mathrm{SL}(n)$. If $A_0\Omega$ minimizes $\Lambda_{2,q}$ among these images, then
\[
T^{\mathcal{A}}_2(\Omega)^{\sigma}\,\Lambda^{\mathcal{A}}_{2,q}(\Omega)\geq T_2(A_0\Omega)^{\sigma}\,\Lambda_{2,q}(A_0\Omega),
\]
so that Theorem \ref{thm:KJ} for $p=2$, including its equality case, follows from \eqref{eq:KJEuclid} applied to $A_0\Omega$. For $p\neq2$ the anisotropic norm $h(\Gamma_pL_u,\cdot)$ is not ellipsoidal in general, and the proof above requires Brasco's anisotropic inequality together with the $L_p$ Busemann--Petty centroid inequality. Conversely, \eqref{eq:KJ} implies \eqref{eq:KJEuclid} for every $\Omega$ with $T^{\mathcal{A}}_p(\Omega)=T_p(\Omega)$, since then the left-hand side of \eqref{eq:KJEuclid} is at least that of \eqref{eq:KJ}.
\end{remark}

\section{Stability of the \texorpdfstring{$L_p$}{Lp} Busemann--Petty centroid inequality}\label{sec:BPstab}

We begin with the stability version of the $L_p$ Petty projection inequality due to B\"or\"oczky.

\begin{theorem}[{\cite[Theorem 4]{Boroczky2013}}]\label{thm:Boroczky}
Let $p\geq1$. There exists $\gamma_0=\gamma_0(n,p)>0$ such that
\[
|M|^{\frac{n-p}{p}}\,|\Pi^*_pM|\leq\Big(1-\gamma_0\,\delta_{\mathrm{BM}}(M,B^n)^{840(n+p)}\Big)\,\omega_n^{\frac np}
\]
for every origin-symmetric convex body $M\subset\mathbb{R}^n$.
\end{theorem}

For $\gamma_0=0$ this is the $L_p$ Petty projection inequality of Lutwak, Yang and Zhang \cite{LYZ2000}. Theorem \ref{thm:Boroczky} is the case $\varphi(t)=|t|^p$ of \cite[Theorem 4]{Boroczky2013}, which is stated for the Orlicz projection bodies of convex bodies containing the origin in their interiors: in that case the error term $\delta^{cn}\,(\varphi(\delta^c)+\varphi(-\delta^c))$, $c=840$, equals $2\delta^{840(n+p)}$. There the result is formulated for the ratio $|\Pi^*_pM|/|M|$ with a normalization of $\Pi_pM$ that differs from \eqref{eq:GammaPi} by a factor depending only on $n$, $p$ and $|M|$; with the normalization \eqref{eq:GammaPi}, this ratio becomes a constant multiple of $|M|^{(n-p)/p}|\Pi^*_pM|$. For origin-symmetric bodies, the Banach--Mazur distance of \cite{Boroczky2013}, which allows translations, coincides with ours.

The class reduction technique rests on the following identity of Lutwak, Yang and Zhang \cite{LYZ2000}; we include the short proof to make the normalizations explicit.

\begin{lemma}\label{lem:classreduction}
Let $p\geq1$, $M\in\mathcal{K}^n_o$ and let $L$ be a star body. Then
\[
V_p(M,\Gamma_pL)=\frac{\omega_n}{|L|}\,\widetilde{V}_{-p}(L,\Pi^*_pM).
\]
\end{lemma}

\begin{proof}
By \eqref{eq:GammaPi} and polar coordinates, $h(\Gamma_pL,v)^p=\frac{1}{n a_{n,p}|L|}\int_{S^{n-1}}\rho_L(\xi)^{n+p}|\langle v,\xi\rangle|^p\,d\xi$. By Fubini's theorem and \eqref{eq:GammaPi} again,
\begin{align*}
V_p(M,\Gamma_pL)&=\frac{1}{n^2a_{n,p}|L|}\int_{S^{n-1}}\rho_L(\xi)^{n+p}\Big(\int_{S^{n-1}}|\langle v,\xi\rangle|^p\,h_M(v)^{1-p}\,dS_M(v)\Big)\,d\xi\\
&=\frac{\omega_n}{n|L|}\int_{S^{n-1}}\rho_L(\xi)^{n+p}\,h(\Pi_pM,\xi)^p\,d\xi,
\end{align*}
and $h(\Pi_pM,\xi)=\rho_{\Pi^*_pM}(\xi)^{-1}$.
\end{proof}

We also need the following quantitative improvement of \eqref{eq:dualmixed}, due to Nguyen.

\begin{lemma}[{\cite[Proposition 2.1]{Nguyen2016}}]\label{lem:Nguyen}
Let $p\geq1$ and let $L$, $Q$ be star bodies. Then
\[
\widetilde{V}_{-p}(L,Q)\geq|L|^{\frac{n+p}{n}}|Q|^{-\frac pn}\Big(1+\frac{p}{8n}\,\eta^2\Big),\qquad\eta=\frac{|L\,\triangle\,(aQ)|}{|L|},\quad a=\Big(\frac{|L|}{|Q|}\Big)^{1/n}.
\]
\end{lemma}

We also recall how the symmetric difference controls the Banach--Mazur distance. The following estimate follows from \cite[(5.3)]{Nguyen2016}, where it is stated without proof as well known; for the sake of completeness, we include a short proof.

\begin{lemma}\label{lem:symdiff}
There exists $c_1=c_1(n)>0$ such that $|L\,\triangle\,Q|\geq c_1\,|L|\,\delta_{\mathrm{BM}}(L,Q)^n$ for all origin-symmetric convex bodies $L,Q\subset\mathbb{R}^n$ with $|L|=|Q|$.
\end{lemma}

\begin{proof}
Both sides are invariant under simultaneous linear images of $L$ and $Q$, so by John's theorem we may assume $B^n\subseteq L\subseteq\sqrt nB^n$. Let $\varepsilon=|L\triangle Q|/|L|$ and let $\beta_n=|\{x\in B^n:\langle x,v\rangle>\frac12\}|$ for $v\in S^{n-1}$. Since $\delta_{\mathrm{BM}}(L,Q)\leq\log n$, it suffices to consider $\varepsilon<\varepsilon_0=\min\{\beta_n\omega_n^{-1}n^{-n/2},(20\sqrt n)^{-n}\}$.

First, $\frac12B^n\subseteq Q$. Otherwise there is $y\in\partial Q$ with $|y|<\frac12$; a supporting hyperplane of $Q$ at $y$ and the symmetry of $Q$ give $w\in S^{n-1}$ with $Q\subseteq\{|\langle x,w\rangle|\leq\frac12\}$, whence $|L\setminus Q|\geq|\{x\in B^n:\langle x,w\rangle>\frac12\}|=\beta_n\geq\varepsilon_0|L|$, a contradiction.

Next, let $x\in Q$ with $\|x\|_L=1+t$, $t>0$. For $b\in\frac12B^n$ and $s\in[0,1]$, the point $(1-s)x+sb$ belongs to $Q$, and since $\|b\|_L\leq|b|\leq\frac12$,
\[
\|(1-s)x+sb\|_L\geq(1-s)(1+t)-\tfrac s2>1\qquad\text{if }s<\tfrac{2t}{2t+3}.
\]
With $s_0=\frac{t}{2t+3}$, the ball $(1-s_0)x+\frac{s_0}2B^n$ is contained in $Q\setminus L$, and therefore $\omega_n(s_0/2)^n\leq\varepsilon|L|\leq\varepsilon\,\omega_nn^{n/2}$, that is, $s_0\leq2\sqrt n\,\varepsilon^{1/n}<\frac1{10}$. Hence $t\leq1$, $s_0\geq t/5$, and $t\leq10\sqrt n\,\varepsilon^{1/n}$; thus $Q\subseteq(1+10\sqrt n\,\varepsilon^{1/n})L$. The same argument with the roles of $L$ and $Q$ exchanged, now using $\frac14B^n\subseteq L$ and $\|b\|_Q\leq2|b|\leq\frac12$ for $b\in\frac14B^n$, gives $L\subseteq(1+20\sqrt n\,\varepsilon^{1/n})Q$. Therefore $\delta_{\mathrm{BM}}(L,Q)\leq30\sqrt n\,\varepsilon^{1/n}$.
\end{proof}

The following Lipschitz-type estimate will also be needed.

\begin{lemma}\label{lem:Lipschitz}
Let $p\geq1$. There exists $C=C(n,p)$ such that $\delta_{\mathrm{BM}}(\Pi^*_pM,B^n)\leq C\,\delta_{\mathrm{BM}}(M,B^n)$ for every origin-symmetric convex body $M\subset\mathbb{R}^n$.
\end{lemma}

\begin{proof}
Fix $x\in S^{n-1}$ and let
\[
\Phi_x(y)=|\langle x,y\rangle|^p\,|y|^{1-p}+(p-1)|y|,\qquad y\in\mathbb{R}^n\setminus\{0\},\qquad\Phi_x(0)=0.
\]
We claim that $\Phi_x$ is convex. The function $G(s,r)=r\,g(s/r)$ with $g(\tau)=|\tau|^p+p-1$ is the perspective of a convex function, hence convex on $\mathbb{R}\times(0,\infty)$, and $\partial_rG(s,r)=(p-1)\big(1-(|s|/r)^p\big)\geq0$ for $r\geq|s|$. Since $\Phi_x(y)=G(\langle x,y\rangle,|y|)$ and $|\langle x,y\rangle|\leq|y|$, for $y=\lambda y_1+(1-\lambda)y_2\neq0$ with $\lambda\in(0,1)$ we obtain
\[
\Phi_x(y)\leq G\big(\langle x,y\rangle,\lambda|y_1|+(1-\lambda)|y_2|\big)\leq\lambda\,\Phi_x(y_1)+(1-\lambda)\,\Phi_x(y_2),
\]
using that $G(\langle x,y\rangle,\cdot)$ is nondecreasing on $[|\langle x,y\rangle|,\infty)$ and the convexity of $G$ (the case where $y_1$ or $y_2$ vanishes is trivial). Hence $\Phi_x$ is the support function of an origin-symmetric convex body $Z_x$, and $|\langle x,v\rangle|^p=h_{Z_x}(v)-(p-1)$ for $v\in S^{n-1}$. Therefore, for every convex body $M$,
\[
\int_{S^{n-1}}|\langle x,v\rangle|^p\,dS_M(v)=n\,V(M[n-1],Z_x)-(p-1)\,n\,V(M[n-1],B^n).
\]

Let $\delta=\delta_{\mathrm{BM}}(M,B^n)$. By the linear covariance of $\Pi^*_p$ and the invariance of $\delta_{\mathrm{BM}}$, we may assume $B^n\subseteq M\subseteq e^\delta B^n$. Then $1\leq h_M\leq e^\delta$ on $S^{n-1}$, and by the monotonicity and homogeneity of mixed volumes, $V(B^n[n-1],Z)\leq V(M[n-1],Z)\leq e^{(n-1)\delta}V(B^n[n-1],Z)$ for $Z\in\{Z_x,B^n\}$. Since $nV(B^n[n-1],Z_x)=\int_{S^{n-1}}\Phi_x\,d\xi=n\omega_n(a_{n,p}+p-1)$ and $nV(B^n[n-1],B^n)=n\omega_n$, we get
\[
n\omega_n\big(a_{n,p}-(p-1)(e^{(n-1)\delta}-1)\big)\leq\int_{S^{n-1}}|\langle x,v\rangle|^p\,dS_M(v)\leq n\omega_n\big(a_{n,p}+(a_{n,p}+p-1)(e^{(n-1)\delta}-1)\big).
\]
Combining this with $e^{-(p-1)\delta}\leq h_M^{1-p}\leq1$ and \eqref{eq:GammaPi}, and writing $\beta=(p-1)/a_{n,p}$, we obtain for every $x\in S^{n-1}$
\[
e^{-(p-1)\delta}\big(1-\beta\,(e^{(n-1)\delta}-1)\big)\leq h(\Pi_pM,x)^p\leq1+(1+\beta)(e^{(n-1)\delta}-1).
\]
Let $\delta_0=\delta_0(n,p)>0$ be such that $\beta\,(e^{(n-1)\delta_0}-1)\leq\frac12$. If $\delta\leq\delta_0$, then, using $1-\tau\geq e^{-2\tau}$ for $\tau\in[0,\frac12]$ and $e^{(n-1)\delta}-1\leq(n-1)e^{(n-1)\delta_0}\delta$, both bounds are of the form $e^{\pm C_1\delta}$ with $C_1=C_1(n,p)$. Hence $e^{-C_1\delta/p}B^n\subseteq\Pi_pM\subseteq e^{C_1\delta/p}B^n$, the same inclusions hold for $\Pi^*_pM$, and $\delta_{\mathrm{BM}}(\Pi^*_pM,B^n)\leq2C_1\delta/p$. If $\delta>\delta_0$, then $\delta_{\mathrm{BM}}(\Pi^*_pM,B^n)\leq\frac12\log n\leq\frac{\log n}{2\delta_0}\,\delta$.
\end{proof}

\begin{proof}[Proof of Theorem \ref{thm:BPstab}]
Let $L$ be an origin-symmetric convex body, let $M=\Gamma_pL$ and $Q=\Pi^*_pM$, both origin-symmetric convex bodies, and set
\[
\delta=\delta_{\mathrm{BM}}(L,B^n),\qquad\delta_1=\delta_{\mathrm{BM}}(M,B^n),\qquad\delta_2=\delta_{\mathrm{BM}}(L,Q),\qquad N=840(n+p).
\]
By Lemma \ref{lem:classreduction} with this $M$, $|M|=V_p(M,M)=\frac{\omega_n}{|L|}\widetilde{V}_{-p}(L,Q)$. By Lemma \ref{lem:Nguyen} and Theorem \ref{thm:Boroczky},
\[
|M|\geq\omega_n\,|L|^{\frac pn}|Q|^{-\frac pn}\Big(1+\frac{p}{8n}\eta^2\Big)\geq|L|^{\frac pn}|M|^{\frac{n-p}{n}}\big(1-\gamma_0\delta_1^N\big)^{-\frac pn}\Big(1+\frac{p}{8n}\eta^2\Big),
\]
with $\eta=|L\triangle(aQ)|/|L|$ and $a=(|L|/|Q|)^{1/n}$. Hence
\[
\frac{|M|}{|L|}\geq\big(1-\gamma_0\delta_1^N\big)^{-1}\Big(1+\frac{p}{8n}\eta^2\Big)^{\frac np}\geq\big(1+\gamma_0\delta_1^N\big)\big(1+c_2\eta^2\big)\geq1+\gamma_0\delta_1^N+c_2\eta^2,
\]
where $c_2=c_2(n,p)>0$; here we used $\eta\leq2$ and $(1+\tau)^{n/p}\geq1+\frac np\frac{\tau}{1+\tau}$ for $\tau\geq0$. Since $|aQ|=|L|$ and $\delta_{\mathrm{BM}}(L,aQ)=\delta_2$, Lemma \ref{lem:symdiff} gives $\eta\geq c_1\delta_2^n$. By Lemma \ref{lem:Lipschitz} and the triangle inequality, $\delta\leq\delta_2+\delta_{\mathrm{BM}}(Q,B^n)\leq\delta_2+C\delta_1$. Since $\delta_2\leq\log n$ and $N\geq2n$, we have $\delta_2^{2n}\geq(\log n)^{2n-N}\delta_2^N$, and therefore
\[
\frac{|\Gamma_pL|}{|L|}-1\geq\gamma_0\delta_1^N+c_2c_1^2\delta_2^{2n}\geq c_3\big(\delta_1^N+\delta_2^N\big)\geq c_3\,2^{1-N}(\delta_1+\delta_2)^N\geq\gamma\,\delta^N,
\]
with $c_3$ and $\gamma$ depending only on $n$ and $p$.
\end{proof}

\section{Stability of the affine inequalities}\label{sec:stability}

We now combine Theorem \ref{thm:BPstab} with the anisotropic structure of $\mathcal{E}_p$ to prove Theorem \ref{thm:stab}.  We first collect the bounds in terms of the polar projection body of an extremal.

\begin{proposition}\label{prop:bodystab}
Let $\gamma=\gamma(n,p)$ be the constant of Theorem \ref{thm:BPstab}, let $N=840(n+p)$, and let $\theta$ and $\kappa$ be as in \eqref{eq:Wulff}.
\begin{enumerate}[label=\textup{(\roman*)}]
\item If $p>1$, $1\leq q<p^*$ and $\Omega\subset\mathbb{R}^n$ is a bounded open set, then, for every minimizer $u$ for $\Lambda^{\mathcal{A}}_{p,q}(\Omega)$,
\[
\Lambda^{\mathcal{A}}_{p,q}(\Omega)\geq\big(1+\gamma\,\delta_{\mathrm{BM}}(\Pi^*_pu,B^n)^N\big)^{\frac pn}\Lambda_{p,q}(\Omega^\star).
\]
\item If $1<p<n$ and $K\subset\mathbb{R}^n$ is a compact set with $|K|>0$, then, for every minimizer $u$ as in Theorem \ref{thm:potential},
\[
\operatorname{cap}^{\mathcal{A}}_p(K)\geq\big(1+\gamma\,\delta_{\mathrm{BM}}(\Pi^*_pu,B^n)^N\big)^{\frac pn}\operatorname{cap}_p(K^\star).
\]
\item If $p>1$, $1<q<p^*$ and $\Omega\subset\mathbb{R}^n$ is a bounded open set, then, for every minimizer $u$ for $\Lambda^{\mathcal{A}}_{p,q}(\Omega)$,
\[
T^{\mathcal{A}}_p(\Omega)^{(p-1)\sigma}\,\Lambda^{\mathcal{A}}_{p,q}(\Omega)\geq\big(1+\gamma\,\delta_{\mathrm{BM}}(\Pi^*_pu,B^n)^N\big)^{\theta}\,\kappa.
\]
\end{enumerate}
\end{proposition}

\begin{proof}
In each case $L_u$ is a dilate of the origin-symmetric convex body $\Pi^*_pu$, and the Banach--Mazur distance is invariant under dilations, so that Theorem \ref{thm:BPstab} gives
\[
\frac{|\Gamma_pL_u|}{|L_u|}\geq1+\gamma\,\delta_{\mathrm{BM}}(\Pi^*_pu,B^n)^N,
\]
and it remains to bound the three quantities from below by the corresponding powers of this ratio times the Euclidean values. For (i), by homogeneity we may assume $\|u\|_q=1$; then $\Lambda^{\mathcal{A}}_{p,q}(\Omega)=\mathcal{E}_p(u)^p$, and \eqref{eq:isochain} with $v=u$ gives the power $p/n$. For (ii), $\operatorname{cap}^{\mathcal{A}}_p(K)=\mathcal{E}_p(u)^p$, and \eqref{eq:capchain} gives the power $p/n$. For (iii), by homogeneity we may assume $\|u\|_q=1$, and then \eqref{eq:KJchain} gives the power $\theta$.
\end{proof}

For $p=2$ these bounds reduce to \eqref{eq:iso}, \eqref{eq:affisocap} and \eqref{eq:KJ}, since $\Pi^*_2u$ is always an ellipsoid (Proposition \ref{prop:p2}).

\begin{proof}[Proof of Theorem \ref{thm:stab}(i) and (ii)]
We begin with (ii). Let $\varepsilon\geq0$ be defined by $\operatorname{cap}^{\mathcal{A}}_p(K)=(1+\varepsilon)\operatorname{cap}_p(K^\star)$. Since $\mathcal{F}_{\mathrm{aff}}(K)\leq2$, we may assume that $\varepsilon\leq1$. For every $A\in\mathrm{SL}(n)$ we have $|AK|=|K|$ and $\mathcal{F}(AK)\geq\mathcal{F}_{\mathrm{aff}}(K)$.

If $p=2$, Theorem \ref{thm:DPMM} gives $\operatorname{cap}_2(AK)\geq\big(1+c_{\mathrm{cap}}\mathcal{F}_{\mathrm{aff}}(K)^2\big)\operatorname{cap}_2(K^\star)$ for every $A\in\mathrm{SL}(n)$, and taking the infimum over $A$, Proposition \ref{prop:compare} proves the claim with $c=c_{\mathrm{cap}}$.

Let $p\neq2$, let $u$ be a minimizer given by Theorem \ref{thm:potential}, and set $\delta=\delta_{\mathrm{BM}}(\Pi^*_pu,B^n)$, $\eta=\delta_{\mathrm{BM}}(\Gamma_pL_u,B^n)$ and $N=840(n+p)$. By Proposition \ref{prop:bodystab}(ii), $(1+\gamma\delta^N)^{p/n}\leq1+\varepsilon$, and since $(1+\varepsilon)^{n/p}-1\leq(2^{n/p}-1)\varepsilon$ for $\varepsilon\in[0,1]$, we get $\delta\leq C_1\varepsilon^{1/N}$. If $E'$ is an origin-symmetric ellipsoid with $E'\subseteq L_u\subseteq e^{\delta}E'$, then \eqref{eq:GammaPi} gives $e^{-\delta n/p}\,\Gamma_pE'\subseteq\Gamma_pL_u\subseteq e^{\delta(n+p)/p}\,\Gamma_pE'$, and $\Gamma_pE'$ is an ellipsoid; hence $\eta\leq\frac{2n+p}{p}\delta\leq C_2\varepsilon^{1/N}$. Let $E=tA_0B^n$, with $t>0$ and $A_0\in\mathrm{SL}(n)$, be an origin-symmetric ellipsoid with $E\subseteq\Gamma_pL_u\subseteq e^{\eta}E$. Then $t^n=|E|/\omega_n\geq e^{-n\eta}|\Gamma_pL_u|/\omega_n\geq e^{-n\eta}$ by the $L_p$ Busemann--Petty centroid inequality. As $H_u=h(\Gamma_pL_u,\cdot)\geq h_E$, Lemma \ref{lem:Wulff}(iii) and Theorem \ref{thm:DPMM} give
\begin{align*}
(1+\varepsilon)\operatorname{cap}_p(K^\star)&=\int_{\mathbb{R}^n}H_u(\nabla u)^p\,dx\geq\operatorname{cap}^{E}_p(K)=t^p\operatorname{cap}_p(A_0^{-1}K)\\
&\geq e^{-p\eta}\big(1+c_{\mathrm{cap}}\,\mathcal{F}(A_0^{-1}K)^2\big)\operatorname{cap}_p(K^\star).
\end{align*}
Consequently $c_{\mathrm{cap}}\,\mathcal{F}_{\mathrm{aff}}(K)^2\leq c_{\mathrm{cap}}\,\mathcal{F}(A_0^{-1}K)^2\leq(1+\varepsilon)e^{pC_2\varepsilon^{1/N}}-1\leq C_3\varepsilon^{1/N}$, which is the claim with $\beta=2N$. The constants $C_1,C_2,C_3$ depend only on $n$ and $p$.

We turn to (i). Let $\varepsilon\geq0$ be defined by $\Lambda^{\mathcal{A}}_{p,q}(\Omega)=(1+\varepsilon)\Lambda_{p,q}(\Omega^\star)$. If $p=2$, Proposition \ref{prop:p2} gives $A\in\mathrm{SL}(n)$ with $\Lambda^{\mathcal{A}}_{2,q}(\Omega)=\Lambda_{2,q}(A\Omega)$; since $|A\Omega|=|\Omega|$ and $\mathcal{F}(A\Omega)\geq\mathcal{F}_{\mathrm{aff}}(\Omega)$, \eqref{eq:quantFK} gives the claim with $\beta=2$. If $p\neq2$, the argument above applies verbatim, with $\Lambda_{p,q}$ in place of $\operatorname{cap}_p$, a minimizer $u$ for $\Lambda^{\mathcal{A}}_{p,q}(\Omega)$ with $\|u\|_q=1$ in place of the potential, Proposition \ref{prop:bodystab}(i) in place of Proposition \ref{prop:bodystab}(ii), the identity $\Lambda^{E}_{p,q}(\Omega)=t^p\Lambda_{p,q}(A_0^{-1}\Omega)$ given by Lemma \ref{lem:Wulff}(iii) in place of its capacitary form, and \eqref{eq:quantFK} in place of Theorem \ref{thm:DPMM}. The case $q=1$ gives the bound for $T^{\mathcal{A}}_p$.
\end{proof}

For the proof of Theorem \ref{thm:stab}(iii) we now compare \eqref{eq:KJ} with the Euclidean inequality \eqref{eq:KJEuclid} by means of Lemma \ref{lem:Wulff}(iii).

\begin{proposition}\label{prop:reduction}
Let $p>1$, $1<q<p^*$, let $\Omega\subset\mathbb{R}^n$ be a bounded open set and let $u$ be a minimizer for $\Lambda^{\mathcal{A}}_{p,q}(\Omega)$. Define $\varepsilon\geq0$ by $T^{\mathcal{A}}_p(\Omega)^{(p-1)\sigma}\Lambda^{\mathcal{A}}_{p,q}(\Omega)=(1+\varepsilon)\,\kappa$, where $\kappa=T_p(B^n)^{(p-1)\sigma}\Lambda_{p,q}(B^n)$, and let $\eta=\delta_{\mathrm{BM}}(\Gamma_pL_u,B^n)$. Then there exists $A\in\mathrm{GL}(n)$ such that
\begin{equation}\label{eq:reduction}
T_p(A^{-1}\Omega)^{(p-1)\sigma}\,\Lambda_{p,q}(A^{-1}\Omega)\leq(1+\varepsilon)\,e^{(p\sigma+n\theta)\eta}\,\kappa,
\end{equation}
where $\theta$ is as in \eqref{eq:Wulff}, and $\eta\leq\frac{2n+p}{p}\,\delta_{\mathrm{BM}}(\Pi^*_pu,B^n)$. In particular, if $\varepsilon\leq1$, the left-hand side of \eqref{eq:reduction} is at most $\big(1+C\varepsilon^{1/(840(n+p))}\big)\kappa$, with $C=C(n,p,q)$. If $p=2$, then $\eta=0$ and $A\in\mathrm{SL}(n)$.
\end{proposition}

\begin{proof}
As in the proof of Proposition \ref{prop:bodystab}(iii), we may assume $u\geq0$ and $\|u\|_q=1$. Let $Q=\Gamma_pL_u$, and let $E=AB^n$, with $\det A>0$, be an origin-symmetric ellipsoid such that $E\subseteq Q\subseteq e^{\eta}E$. Since $h_E\leq h_Q\leq e^{\eta}h_E$, we have $\Lambda^{Q}_{p,q}(\Omega)\geq\Lambda^{E}_{p,q}(\Omega)$ and $\Lambda^{Q}_{p,1}(\Omega)\leq e^{p\eta}\Lambda^{E}_{p,1}(\Omega)$, that is, $T^{Q}_p(\Omega)\geq e^{-p\eta/(p-1)}T^{E}_p(\Omega)$. Combining the first inequality in \eqref{eq:KJchain} with Lemma \ref{lem:Wulff}(iii) and the identity $(p-1)\sigma+1-\frac pq=\theta$, we obtain
\[
(1+\varepsilon)\,\kappa\geq T^{Q}_p(\Omega)^{(p-1)\sigma}\,\Lambda^{Q}_{p,q}(\Omega)\geq e^{-p\sigma\eta}\,(\det A)^{\theta}\,T_p(A^{-1}\Omega)^{(p-1)\sigma}\,\Lambda_{p,q}(A^{-1}\Omega).
\]
Since $Q\subseteq e^{\eta}E$ and $|Q|\geq|L_u|=\omega_n$ by \eqref{eq:BP}, we have $\det A=|E|/\omega_n\geq e^{-n\eta}$, and \eqref{eq:reduction} follows. If $E'\subseteq L_u\subseteq e^{s}E'$ for an origin-symmetric ellipsoid $E'$, then \eqref{eq:GammaPi} gives $e^{-sn/p}\,\Gamma_pE'\subseteq\Gamma_pL_u\subseteq e^{s(n+p)/p}\,\Gamma_pE'$, and $\Gamma_pE'$ is an ellipsoid; hence $\eta\leq\frac{2n+p}{p}\,\delta_{\mathrm{BM}}(L_u,B^n)=\frac{2n+p}{p}\,\delta_{\mathrm{BM}}(\Pi^*_pu,B^n)$. For $\varepsilon\leq1$, Proposition \ref{prop:bodystab}(iii) gives $\delta_{\mathrm{BM}}(\Pi^*_pu,B^n)\leq C\varepsilon^{1/(840(n+p))}$, and the bound on the left-hand side of \eqref{eq:reduction} follows. Finally, for $p=2$ the body $\Pi^*_2u$ is an ellipsoid, so $L_u$ is an origin-symmetric ellipsoid of volume $\omega_n$ and $Q=\Gamma_2L_u=L_u$; we may take $E=Q$, so that $\eta=0$ and $\det A=1$.
\end{proof}

Since $A^{-1}\Omega$ is close to a ball precisely when $\Omega$ is close to the image of that ball under $A$, Proposition \ref{prop:reduction} transfers every quantitative form of \eqref{eq:KJEuclid} on a class of sets invariant under linear maps to \eqref{eq:KJ}, as stated in the introduction. The class cannot be that of all open sets.

\begin{remark}\label{rem:dust}
Let $V>0$, let $W$ be a cube with $|W|=2^{n+1}V/\omega_n$ at distance $R$ from $B^n$, and for small $r>0$ let $D_r\subseteq W$ be the union of $N_r=\lfloor V/(\omega_nr^n)\rfloor$ pairwise disjoint balls of radius $r$, for instance centered at points of $2r\mathbb{Z}^n$. Let $\Omega_r=B^n\cup D_r$. For every origin-symmetric convex body $Q$, the functional $T^{Q}_p$ is subadditive on disjoint unions of open sets, as follows from H\"older's inequality applied to $v=v_1+v_2$ with $v_i\in W^{1,p}_0(\Omega_i)$. By \eqref{eq:envelope}, $T^{\mathcal{A}}_p(\Omega)$ is the supremum of $T^{\Gamma_pL}_p(\Omega)$ over star bodies $L$ with $|L|=\omega_n$, and $T^{\Gamma_pL}_p(B_r)\leq T^{\mathcal{A}}_p(B_r)=T_p(B_r)=r^{n+p'}\,T_p(B^n)$ by Corollary \ref{cor:SV}. Hence
\[
T_p(B^n)\leq T^{\mathcal{A}}_p(\Omega_r)\leq T_p(B^n)+N_r\,T_p(B_r)\leq\Big(1+\frac{V}{\omega_n}\,r^{p'}\Big)\,T_p(B^n),
\]
while $\Lambda^{\mathcal{A}}_{p,q}(\Omega_r)\leq\Lambda^{\mathcal{A}}_{p,q}(B^n)=\Lambda_{p,q}(B^n)$. Therefore the deficit in \eqref{eq:KJ} for $\Omega_r$ tends to $0$ as $r\to0$, and the same argument applies to \eqref{eq:KJEuclid}. On the other hand, an ellipsoid $E$ with $|E\triangle\Omega_r|<\frac14\min\{\omega_n,V\}$ contains subsets of $B^n$ and of $D_r$ of volume at least $\frac14\min\{\omega_n,V\}$ for $r$ small, and hence their convex hull, whose volume tends to infinity with $R$ by the Brunn--Minkowski inequality. Choosing $R$ large, depending only on $n$ and $V$, we obtain $|E\triangle\Omega_r|\geq\frac14\min\{\omega_n,V\}$ for every ellipsoid $E$ and every small $r$. Thus, since $|\Omega_r|\leq\omega_n+V$, we have $\mathcal{F}_{\mathrm{aff}}(\Omega_r)\geq\min\{\omega_n,V\}/\big(4(\omega_n+V)\big)$ for all small $r$, while the deficit in \eqref{eq:KJ} tends to $0$: for general open sets, the deficit does not control the affine Fraenkel asymmetry.
\end{remark}

No such example exists for the affine capacity, in accordance with Theorem \ref{thm:stab}: a ball of radius $r$ has torsional rigidity proportional to $r^{n+p'}$ but capacity proportional to $r^{n-p}$, so that many small balls of fixed total volume contribute little to the torsional rigidity and much to the capacity.

\begin{proof}[Proof of Theorem \ref{thm:stab}(iii)]
Let $\kappa$ be as in \eqref{eq:Wulff} and write $P(\Omega)=T^{\mathcal{A}}_p(\Omega)^{(p-1)\sigma}\Lambda^{\mathcal{A}}_{p,q}(\Omega)$. By \eqref{eq:invariance}, $P$ is invariant under affine maps with positive determinant, and so is the conclusion of (iii). Moreover, $T^{\mathcal{A}}_p$ is nondecreasing and $\Lambda^{\mathcal{A}}_{p,q}$ is nonincreasing with respect to inclusion. Hence, if $(1-\tau)\Omega'\subseteq\Omega\subseteq(1+\tau)\Omega'$ for an open set $\Omega'$ and $\tau\in(0,1)$, then, since $(p-1)\sigma(n+p')=p+\frac{np}{q}-n>0$,
\begin{equation}\label{eq:sandwich}
P(\Omega)\geq T^{\mathcal{A}}_p\big((1-\tau)\Omega'\big)^{(p-1)\sigma}\,\Lambda^{\mathcal{A}}_{p,q}\big((1+\tau)\Omega'\big)=\Big(\frac{1-\tau}{1+\tau}\Big)^{p+\frac{np}{q}-n}P(\Omega').
\end{equation}

Suppose that (iii) fails for some $\eta>0$. Then there are bounded open convex sets $\Omega_k$ with $P(\Omega_k)\leq(1+\frac1k)\kappa$ for which no ellipsoid has the stated property. By John's theorem \cite{John1948} and the affine invariance, we may assume $B^n\subseteq\overline{\Omega_k}\subseteq nB^n$. By the Blaschke selection theorem, a subsequence of $(\overline{\Omega_k})$ converges in the Hausdorff metric to a convex body $K_\infty$ with $B^n\subseteq K_\infty\subseteq nB^n$. Since all these bodies contain $B^n$, their radial functions converge uniformly, so that $(1-\tau_k)\Omega_\infty\subseteq\Omega_k\subseteq(1+\tau_k)\Omega_\infty$, where $\Omega_\infty=\operatorname{int}K_\infty$ and $\tau_k\to0$. By \eqref{eq:sandwich}, $P(\Omega_\infty)\leq\kappa$, so equality holds in \eqref{eq:KJ} for $\Omega_\infty$, and by Theorem \ref{thm:KJ} there is an open ellipsoid $E$ with $\Omega_\infty\subseteq E$ and $C_p(E\setminus\Omega_\infty)=0$. Since sets of zero Sobolev $p$-capacity have measure zero, $E\setminus\Omega_\infty$ has empty interior, so $\Omega_\infty$ is dense in $E$. Hence, $K_\infty\supseteq E$ and $\Omega_\infty=\operatorname{int}K_\infty\supseteq E$; hence $\Omega_\infty=E$. Let $z$ be the center of $E$. Since $B^n\subseteq E\subseteq nB^n$ and $E$ is symmetric about $z$, we have $B^n\subseteq E-z$ and $|z|\leq n$. Therefore the ellipsoid $E_k=(1-\tau_k)E$, whose center is $z_k=(1-\tau_k)z$, satisfies
\[
E_k\subseteq\Omega_k\subseteq(1+\tau_k)E\subseteq z_k+\frac{1+(2n+1)\tau_k}{1-\tau_k}\,(E_k-z_k),
\]
which contradicts the choice of $\Omega_k$ for $k$ large.

Finally, let $E$ and $z$ be as in the conclusion of (iii). Since $|E|\leq|\Omega|\leq e^{n\eta}|E|$, the ellipsoid $E_0=z+s(E-z)$ with $s=(|\Omega|/|E|)^{1/n}\in[1,e^{\eta}]$ satisfies $|E_0|=|\Omega|$ and $E\subseteq E_0\subseteq z+e^{\eta}(E-z)$. As $\Omega$ lies between the same two sets, $|\Omega\triangle E_0|\leq\big|\big(z+e^{\eta}(E-z)\big)\setminus E\big|=(e^{n\eta}-1)|E|\leq(e^{n\eta}-1)|\Omega|$, and therefore $\mathcal{F}_{\mathrm{aff}}(\Omega)\leq e^{n\eta}-1$.
\end{proof}

\textbf{AI-assisted tools.} The author used a large language model as an assistant during the preparation of this manuscript, mainly for checking computations and constants, checking the literature and the bibliography, and polishing the manuscript. The author takes full responsibility for the content of the paper.

\medskip
\noindent\textbf{Conflicts of interest.} There are no conflicts of interest.

\bibliographystyle{acm}
\bibliography{affine-merged}

@book {AdamsHedberg1996,
    AUTHOR = {Adams, David R. and Hedberg, Lars Inge},
     TITLE = {Function spaces and potential theory},
    SERIES = {Grundlehren der mathematischen Wissenschaften},
    VOLUME = {314},
 PUBLISHER = {Springer-Verlag, Berlin},
      YEAR = {1996},
}

@article {AFLT1997,
    AUTHOR = {Alvino, Angelo and Ferone, Vincenzo and Lions, Pierre-Louis and Trombetti, Guido},
     TITLE = {Convex symmetrization and applications},
   JOURNAL = {Ann. Inst. H. Poincar\'{e} Anal. Non Lin\'{e}aire},
  FJOURNAL = {Annales de l'Institut Henri Poincar\'{e}. Analyse Non Lin\'{e}aire},
    VOLUME = {14},
      YEAR = {1997},
    NUMBER = {2},
     PAGES = {275--293},
  MRNUMBER = {1441395},
}

@article {Aubin1976,
    AUTHOR = {Aubin, Thierry},
     TITLE = {Probl\`emes isop\'{e}rim\'{e}triques et espaces de {S}obolev},
   JOURNAL = {J. Differential Geometry},
  FJOURNAL = {Journal of Differential Geometry},
    VOLUME = {11},
      YEAR = {1976},
    NUMBER = {4},
     PAGES = {573--598},
}

@article {Boroczky2013,
    AUTHOR = {B\"{o}r\"{o}czky, K\'{a}roly J.},
     TITLE = {Stronger versions of the {O}rlicz-{P}etty projection inequality},
   JOURNAL = {J. Differential Geom.},
  FJOURNAL = {Journal of Differential Geometry},
    VOLUME = {95},
      YEAR = {2013},
    NUMBER = {2},
     PAGES = {215--247},
       DOI = {10.4310/jdg/1376053446},
}

@article {Brasco2014,
    AUTHOR = {Brasco, Lorenzo},
     TITLE = {On torsional rigidity and principal frequencies: an invitation to the {K}ohler-{J}obin rearrangement technique},
   JOURNAL = {ESAIM Control Optim. Calc. Var.},
  FJOURNAL = {ESAIM. Control, Optimisation and Calculus of Variations},
    VOLUME = {20},
      YEAR = {2014},
    NUMBER = {2},
     PAGES = {315--338},
  MRNUMBER = {3264206},
       DOI = {10.1051/cocv/2013065},
}

@article {BrothersZiemer1988,
    AUTHOR = {Brothers, John E. and Ziemer, William P.},
     TITLE = {Minimal rearrangements of {S}obolev functions},
   JOURNAL = {J. Reine Angew. Math.},
  FJOURNAL = {Journal f\"{u}r die Reine und Angewandte Mathematik},
    VOLUME = {384},
      YEAR = {1988},
     PAGES = {153--179},
}

@article {CG2002,
    AUTHOR = {Campi, Stefano and Gronchi, Paolo},
     TITLE = {The {$L^p$}-{B}usemann-{P}etty centroid inequality},
   JOURNAL = {Adv. Math.},
  FJOURNAL = {Advances in Mathematics},
    VOLUME = {167},
      YEAR = {2002},
    NUMBER = {1},
     PAGES = {128--141},
}

@article {CLYZ2009,
    AUTHOR = {Cianchi, Andrea and Lutwak, Erwin and Yang, Deane and Zhang, Gaoyong},
     TITLE = {Affine {M}oser-{T}rudinger and {M}orrey-{S}obolev inequalities},
   JOURNAL = {Calc. Var. Partial Differential Equations},
  FJOURNAL = {Calculus of Variations and Partial Differential Equations},
    VOLUME = {36},
      YEAR = {2009},
    NUMBER = {3},
     PAGES = {419--436},
  MRNUMBER = {2551138},
}

@article {DPMM2021,
    AUTHOR = {De Philippis, Guido and Marini, Michele and Mukoseeva, Ekaterina},
     TITLE = {The sharp quantitative isocapacitary inequality},
   JOURNAL = {Rev. Mat. Iberoam.},
  FJOURNAL = {Revista Matem\'{a}tica Iberoamericana},
    VOLUME = {37},
      YEAR = {2021},
    NUMBER = {6},
     PAGES = {2191--2228},
       DOI = {10.4171/rmi/1259},
}

@article {FMP2009,
    AUTHOR = {Fusco, Nicola and Maggi, Francesco and Pratelli, Aldo},
     TITLE = {Stability estimates for certain {F}aber-{K}rahn, isocapacitary and {C}heeger inequalities},
   JOURNAL = {Ann. Sc. Norm. Super. Pisa Cl. Sci. (5)},
  FJOURNAL = {Annali della Scuola Normale Superiore di Pisa. Classe di Scienze. Serie V},
    VOLUME = {8},
      YEAR = {2009},
    NUMBER = {1},
     PAGES = {51--71},
}

@article {FranzinaLamberti2010,
    AUTHOR = {Franzina, Giovanni and Lamberti, Pier Domenico},
     TITLE = {Existence and uniqueness for a {$p$}-{L}aplacian nonlinear eigenvalue problem},
   JOURNAL = {Electron. J. Differential Equations},
  FJOURNAL = {Electronic Journal of Differential Equations},
      YEAR = {2010},
     PAGES = {Paper No. 26, 10},
}

@book {Gardner2006,
    AUTHOR = {Gardner, Richard J.},
     TITLE = {Geometric tomography},
   EDITION = {Second},
    SERIES = {Encyclopedia of Mathematics and its Applications},
    VOLUME = {58},
 PUBLISHER = {Cambridge University Press, New York},
      YEAR = {2006},
}

@article {HHW1991,
    AUTHOR = {Hall, R. R. and Hayman, W. K. and Weitsman, A. W.},
     TITLE = {On asymmetry and capacity},
   JOURNAL = {J. Analyse Math.},
  FJOURNAL = {Journal d'Analyse Math\'{e}matique},
    VOLUME = {56},
      YEAR = {1991},
     PAGES = {87--123},
}

@article {HJM2016,
    AUTHOR = {Haddad, Juli\'{a}n and Jim\'{e}nez, C. Hugo and Montenegro, Marcos},
     TITLE = {Sharp affine {S}obolev type inequalities via the {$L_p$} {B}usemann-{P}etty centroid inequality},
   JOURNAL = {J. Funct. Anal.},
  FJOURNAL = {Journal of Functional Analysis},
    VOLUME = {271},
      YEAR = {2016},
    NUMBER = {2},
     PAGES = {454--473},
}

@article {HJM2021,
    AUTHOR = {Haddad, Juli\'{a}n and Jim\'{e}nez, C. Hugo and Montenegro, Marcos},
     TITLE = {From affine {P}oincar\'{e} inequalities to affine spectral inequalities},
   JOURNAL = {Adv. Math.},
  FJOURNAL = {Advances in Mathematics},
    VOLUME = {386},
      YEAR = {2021},
     PAGES = {Paper No. 107808, 35},
}

@book {HKM2006,
    AUTHOR = {Heinonen, Juha and Kilpel\"{a}inen, Tero and Martio, Olli},
     TITLE = {Nonlinear potential theory of degenerate elliptic equations},
 PUBLISHER = {Dover Publications, Inc., Mineola, NY},
      YEAR = {2006},
      NOTE = {Unabridged republication of the 1993 original},
}

@article {HSX2012,
    AUTHOR = {Haberl, Christoph and Schuster, Franz E. and Xiao, Jie},
     TITLE = {An asymmetric affine {P}\'{o}lya-{S}zeg\"{o} principle},
   JOURNAL = {Math. Ann.},
  FJOURNAL = {Mathematische Annalen},
    VOLUME = {352},
      YEAR = {2012},
    NUMBER = {3},
     PAGES = {517--542},
  MRNUMBER = {2885586},
}

@incollection {John1948,
    AUTHOR = {John, Fritz},
     TITLE = {Extremum problems with inequalities as subsidiary conditions},
 BOOKTITLE = {Studies and {E}ssays {P}resented to {R}. {C}ourant on his 60th {B}irthday, {J}anuary 8, 1948},
     PAGES = {187--204},
 PUBLISHER = {Interscience Publishers, Inc., New York},
      YEAR = {1948},
}

@article {KohlerJobin1975,
    AUTHOR = {Kohler-Jobin, Marie-Th\'{e}r\`ese},
     TITLE = {D\'{e}monstration de l'in\'{e}galit\'{e} isop\'{e}rim\'{e}trique {$P\lambda^2\geq\pi j_0^4/2$}, conjectur\'{e}e par {P}\'{o}lya et {S}zeg{\H{o}}},
   JOURNAL = {C. R. Acad. Sci. Paris S\'{e}r. A-B},
    VOLUME = {281},
      YEAR = {1975},
    NUMBER = {2-3},
     PAGES = {A119--A121},
  MRNUMBER = {385691},
}

@article {KohlerJobin1978,
    AUTHOR = {Kohler-Jobin, Marie-Th\'{e}r\`ese},
     TITLE = {Une m\'{e}thode de comparaison isop\'{e}rim\'{e}trique de fonctionnelles de domaines de la physique math\'{e}matique. {I}. {U}ne d\'{e}monstration de la conjecture isop\'{e}rim\'{e}trique {$P\lambda^2\geq\pi j_0^4/2$} de {P}\'{o}lya et {S}zeg{\H{o}}},
   JOURNAL = {Z. Angew. Math. Phys.},
  FJOURNAL = {Zeitschrift f\"{u}r Angewandte Mathematik und Physik},
    VOLUME = {29},
      YEAR = {1978},
    NUMBER = {5},
     PAGES = {757--766},
  MRNUMBER = {511908},
}

@article {KohlerJobin1982,
    AUTHOR = {Kohler-Jobin, Marie-Th\'{e}r\`ese},
     TITLE = {Symmetrization with equal {D}irichlet integrals},
   JOURNAL = {SIAM J. Math. Anal.},
  FJOURNAL = {SIAM Journal on Mathematical Analysis},
    VOLUME = {13},
      YEAR = {1982},
    NUMBER = {1},
     PAGES = {153--161},
  MRNUMBER = {641547},
}

@article {LM2022,
    AUTHOR = {Leite, Edir J. F. and Montenegro, Marcos},
     TITLE = {Minimization to the {Z}hang's energy on {$BV(\Omega)$} and sharp affine {P}oincar\'{e}-{S}obolev inequalities},
   JOURNAL = {J. Funct. Anal.},
  FJOURNAL = {Journal of Functional Analysis},
    VOLUME = {283},
      YEAR = {2022},
    NUMBER = {10},
     PAGES = {Paper No. 109646, 24},
}

@article {LM2024ACV,
    AUTHOR = {Leite, Edir J. F. and Montenegro, Marcos},
     TITLE = {Least energy solutions for affine {$p$}-{L}aplace equations involving subcritical and critical nonlinearities},
   JOURNAL = {Adv. Calc. Var.},
  FJOURNAL = {Advances in Calculus of Variations},
    VOLUME = {17},
      YEAR = {2024},
    NUMBER = {3},
     PAGES = {629--641},
}

@article {LM2024CV,
    AUTHOR = {Leite, Edir J. F. and Montenegro, Marcos},
     TITLE = {Towards existence theorems to affine {$p$}-{L}aplace equations via variational approach},
   JOURNAL = {Calc. Var. Partial Differential Equations},
  FJOURNAL = {Calculus of Variations and Partial Differential Equations},
    VOLUME = {63},
      YEAR = {2024},
    NUMBER = {3},
     PAGES = {Paper No. 73, 23},
}

@misc {LR2025,
    AUTHOR = {Langharst, Dylan and Roysdon, Michael},
     TITLE = {Affine isoperimetric inequalities for the first eigenvalue of the $m$-th order affine $p$-{L}aplace operator},
      NOTE = {Preprint, arXiv:2512.17237},
      YEAR = {2025},
}

@article {Lutwak1993,
    AUTHOR = {Lutwak, Erwin},
     TITLE = {The {B}runn-{M}inkowski-{F}irey theory. {I}. {M}ixed volumes and the {M}inkowski problem},
   JOURNAL = {J. Differential Geom.},
  FJOURNAL = {Journal of Differential Geometry},
    VOLUME = {38},
      YEAR = {1993},
    NUMBER = {1},
     PAGES = {131--150},
  MRNUMBER = {1231704},
}

@article {LYZ2000,
    AUTHOR = {Lutwak, Erwin and Yang, Deane and Zhang, Gaoyong},
     TITLE = {{$L_p$} affine isoperimetric inequalities},
   JOURNAL = {J. Differential Geom.},
  FJOURNAL = {Journal of Differential Geometry},
    VOLUME = {56},
      YEAR = {2000},
    NUMBER = {1},
     PAGES = {111--132},
  MRNUMBER = {1863023},
}

@article {LYZ2002,
    AUTHOR = {Lutwak, Erwin and Yang, Deane and Zhang, Gaoyong},
     TITLE = {Sharp affine {$L_p$} {S}obolev inequalities},
   JOURNAL = {J. Differential Geom.},
  FJOURNAL = {Journal of Differential Geometry},
    VOLUME = {62},
      YEAR = {2002},
    NUMBER = {1},
     PAGES = {17--38},
  MRNUMBER = {1987375},
}

@article {LYZ2006,
    AUTHOR = {Lutwak, Erwin and Yang, Deane and Zhang, Gaoyong},
     TITLE = {Optimal {S}obolev norms and the {$L^p$} {M}inkowski problem},
   JOURNAL = {Int. Math. Res. Not.},
  FJOURNAL = {International Mathematics Research Notices},
      YEAR = {2006},
     PAGES = {Art. ID 62987, 21},
}

@book {Mazya2011,
    AUTHOR = {Maz'ya, Vladimir},
     TITLE = {Sobolev spaces with applications to elliptic partial differential equations},
   EDITION = {Second, revised and augmented},
    SERIES = {Grundlehren der mathematischen Wissenschaften},
    VOLUME = {342},
 PUBLISHER = {Springer, Heidelberg},
      YEAR = {2011},
}

@article {Mukoseeva2023,
    AUTHOR = {Mukoseeva, Ekaterina},
     TITLE = {The sharp quantitative isocapacitary inequality (the case of {$p$}-capacity)},
   JOURNAL = {Adv. Calc. Var.},
  FJOURNAL = {Advances in Calculus of Variations},
    VOLUME = {16},
      YEAR = {2023},
    NUMBER = {1},
     PAGES = {131--162},
       DOI = {10.1515/acv-2020-0106},
}

@article {Nguyen2016,
    AUTHOR = {Nguyen, Van Hoang},
     TITLE = {New approach to the affine {P}\'{o}lya-{S}zeg\"{o} principle and the stability version of the affine {S}obolev inequality},
   JOURNAL = {Adv. Math.},
  FJOURNAL = {Advances in Mathematics},
    VOLUME = {302},
      YEAR = {2016},
     PAGES = {1080--1110},
}

@article {Polya1948,
    AUTHOR = {P\'{o}lya, George},
     TITLE = {Torsional rigidity, principal frequency, electrostatic capacity and symmetrization},
   JOURNAL = {Quart. Appl. Math.},
  FJOURNAL = {Quarterly of Applied Mathematics},
    VOLUME = {6},
      YEAR = {1948},
     PAGES = {267--277},
}

@book {PolyaSzego1951,
    AUTHOR = {P\'{o}lya, George and Szeg\H{o}, Gabor},
     TITLE = {Isoperimetric inequalities in mathematical physics},
    SERIES = {Annals of Mathematics Studies},
    VOLUME = {27},
 PUBLISHER = {Princeton University Press, Princeton, NJ},
      YEAR = {1951},
}

@article {SchindlerTintarev2018,
    AUTHOR = {Schindler, Ian and Tintarev, Cyril},
     TITLE = {Compactness properties and ground states for the affine {L}aplacian},
   JOURNAL = {Calc. Var. Partial Differential Equations},
  FJOURNAL = {Calculus of Variations and Partial Differential Equations},
    VOLUME = {57},
      YEAR = {2018},
    NUMBER = {2},
     PAGES = {Paper No. 48, 14},
}

@book {Schneider2014,
    AUTHOR = {Schneider, Rolf},
     TITLE = {Convex bodies: the {B}runn-{M}inkowski theory},
   EDITION = {Second expanded},
    SERIES = {Encyclopedia of Mathematics and its Applications},
    VOLUME = {151},
 PUBLISHER = {Cambridge University Press, Cambridge},
      YEAR = {2014},
}

@article {Talenti1976,
    AUTHOR = {Talenti, Giorgio},
     TITLE = {Best constant in {S}obolev inequality},
   JOURNAL = {Ann. Mat. Pura Appl. (4)},
  FJOURNAL = {Annali di Matematica Pura ed Applicata. Serie Quarta},
    VOLUME = {110},
      YEAR = {1976},
     PAGES = {353--372},
}

@article {Wang2013,
    AUTHOR = {Wang, Tuo},
     TITLE = {The affine {P}\'{o}lya-{S}zeg\"{o} principle: equality cases and stability},
   JOURNAL = {J. Funct. Anal.},
  FJOURNAL = {Journal of Functional Analysis},
    VOLUME = {265},
      YEAR = {2013},
    NUMBER = {8},
     PAGES = {1728--1748},
       DOI = {10.1016/j.jfa.2013.06.001},
}

@article {WangXiao2020,
    AUTHOR = {Wang, Tuo and Xiao, Jie},
     TITLE = {A new capacity for the affine bounded variation},
   JOURNAL = {Adv. in Appl. Math.},
  FJOURNAL = {Advances in Applied Mathematics},
    VOLUME = {121},
      YEAR = {2020},
     PAGES = {Paper No. 102084, 27},
}

@article {Xiao2016,
    AUTHOR = {Xiao, Jie},
     TITLE = {The {$p$}-affine capacity},
   JOURNAL = {J. Geom. Anal.},
  FJOURNAL = {Journal of Geometric Analysis},
    VOLUME = {26},
      YEAR = {2016},
    NUMBER = {2},
     PAGES = {947--966},
}

@article {Xiao2017,
    AUTHOR = {Xiao, Jie},
     TITLE = {The {$p$}-affine capacity redux},
   JOURNAL = {J. Geom. Anal.},
  FJOURNAL = {Journal of Geometric Analysis},
    VOLUME = {27},
      YEAR = {2017},
    NUMBER = {4},
     PAGES = {2872--2888},
}

@article {XiaoZhang2016,
    AUTHOR = {Xiao, Jie and Zhang, Ning},
     TITLE = {The relative {$p$}-affine capacity},
   JOURNAL = {Proc. Amer. Math. Soc.},
  FJOURNAL = {Proceedings of the American Mathematical Society},
    VOLUME = {144},
      YEAR = {2016},
    NUMBER = {8},
     PAGES = {3537--3554},
       DOI = {10.1090/proc/12980},
}

@article {Zhang1999,
    AUTHOR = {Zhang, Gaoyong},
     TITLE = {The affine {S}obolev inequality},
   JOURNAL = {J. Differential Geom.},
  FJOURNAL = {Journal of Differential Geometry},
    VOLUME = {53},
      YEAR = {1999},
    NUMBER = {1},
     PAGES = {183--202},
  MRNUMBER = {1776095},
}

@article {HuangLi2016,
    AUTHOR = {Huang, Qingzhong and Li, Ai-Jun},
     TITLE = {Optimal {S}obolev norms in the affine class},
   JOURNAL = {J. Math. Anal. Appl.},
  FJOURNAL = {Journal of Mathematical Analysis and Applications},
    VOLUME = {436},
      YEAR = {2016},
    NUMBER = {1},
     PAGES = {568--585},
       DOI = {10.1016/j.jmaa.2015.11.063},
}

@article {Ivaki2016,
    AUTHOR = {Ivaki, Mohammad N.},
     TITLE = {The planar {B}usemann-{P}etty centroid inequality and its stability},
   JOURNAL = {Trans. Amer. Math. Soc.},
  FJOURNAL = {Transactions of the American Mathematical Society},
    VOLUME = {368},
      YEAR = {2016},
    NUMBER = {5},
     PAGES = {3539--3563},
}

@article {HongYe2018,
    AUTHOR = {Hong, Han and Ye, Deping},
     TITLE = {Sharp geometric inequalities for the general {$p$}-affine capacity},
   JOURNAL = {J. Geom. Anal.},
  FJOURNAL = {Journal of Geometric Analysis},
    VOLUME = {28},
      YEAR = {2018},
    NUMBER = {3},
     PAGES = {2254--2287},
}

@misc {ZhouYe2025,
    AUTHOR = {Zhou, X. and Ye, Deping},
     TITLE = {On the {$m$}th order {$p$}-affine capacity},
      NOTE = {Preprint, arXiv:2505.12573},
      YEAR = {2025},
}

@article {BDV2015,
    AUTHOR = {Brasco, Lorenzo and De Philippis, Guido and Velichkov, Bozhidar},
     TITLE = {Faber-{K}rahn inequalities in sharp quantitative form},
   JOURNAL = {Duke Math. J.},
  FJOURNAL = {Duke Mathematical Journal},
    VOLUME = {164},
      YEAR = {2015},
    NUMBER = {9},
     PAGES = {1777--1831},
}

@article {FuscoZhang2017,
    AUTHOR = {Fusco, Nicola and Zhang, Yi Ru-Ya},
     TITLE = {A quantitative form of {F}aber-{K}rahn inequality},
   JOURNAL = {Calc. Var. Partial Differential Equations},
  FJOURNAL = {Calculus of Variations and Partial Differential Equations},
    VOLUME = {56},
      YEAR = {2017},
    NUMBER = {5},
     PAGES = {Paper No. 138, 44},
       DOI = {10.1007/s00526-017-1224-7},
}

@incollection {BrascoDePhilippis2017,
    AUTHOR = {Brasco, Lorenzo and De Philippis, Guido},
     TITLE = {Spectral inequalities in quantitative form},
 BOOKTITLE = {Shape optimization and spectral theory},
     PAGES = {201--281},
 PUBLISHER = {De Gruyter Open, Warsaw},
      YEAR = {2017},
}
\end{document}